\documentclass[final,times]{elsarticle}
\usepackage{commath}
\usepackage[english]{babel}
\usepackage{amsmath,amsthm}
\usepackage{amsfonts}
\usepackage{mathscinet}
\usepackage{textcomp}
\usepackage{units}
\usepackage{enumerate}
\usepackage{srcltx}
\usepackage{graphics}
\usepackage{color}
\usepackage{mathrsfs}
\usepackage{amsthm}
\usepackage{rotating}
\usepackage{tikz}
\usepackage{graphicx}
\usepackage[bookmarks=true, pdfstartview={FitH}]{hyperref}
\usepackage[margin=1in]{geometry}

\usepackage[arrow,matrix,curve]{xy}
\newdir{ >}{{}*!/-1em/\dir{>}}

\newcommand{\lra}{\longrightarrow}

\newcommand{\os}{\overset}

\newcommand{\llm}{\underset{\longleftarrow}{\lim}\,}

\newcommand{\Hom}{\operatorname{Hom}}
\newcommand{\Ext}{\operatorname{Ext}}

\newenvironment{dedication}
{\vspace{6ex}\begin{quotation}\begin{center}\begin{em}}
			{\par\end{em}\end{center}\end{quotation}}

\theoremstyle{dgthm}
\newtheorem{theorem}{Theorem}
\newtheorem{corollary}{Corollary}
\newtheorem{lemma}{Lemma}

\theoremstyle{dgdef}
\newtheorem{definition}{Definition}

\begin{document}
\begin{dedication}
	\hspace{\fill}
	\vspace*{0.1cm}{Dedicated to the memory of academician George Chogoshvili}
\end{dedication}

\title{On Axiomatic Characterization of Alexander-Spanier Normal Homology Theory of General Topological Spaces}

\author{Vladimer Baladze, Anzor Beridze, Leonard Mdzinarishvili }

\address{Department of Mathematics,
Faculty of exact sciences and education,
	Batumi Shota Rustaveli State University,
	35, Ninoshvili St., Batumi,
	Georgia; e-mail:~vladimer.baladze@bsu.edu.ge}

\address{School of Mathematics,
	Kutaisi International University, Youth Avenue, 5th Lane,
	Kutaisi, 4600 Georgia; e-mail:~anzor.beridze@kiu.edu.ge}

\address{Department of Mathematics,
Faculty of Informatics and Control Systems,
Georgian Technical University,
77, Kostava St., Tbilisi,
Georgia; e-mail:~l.mdzinarishvili@gtu.ge}

\begin{abstract} The  Alexandroff-\v{C}ech normal cohomology theory \cite{Mor}, \cite{Bar}, \cite{Ba1},\cite{Ba2} is the unique continuous extension \cite{Wat} of the additive cohomology theory \cite{Mil}, \cite{BM1} from the category of polyhedral pairs $\mathcal{K}^2_{Pol}$ to the category of closed normally embedded, the so called, $P$-pairs of general topological spaces $\mathcal{K}^2_{Top}$. In this paper we define the Alexander-Spanier normal cohomology theory based on all normal coverings and show that it is isomorphic to the Alexandroff-\v{C}ech normal cohomology. Using this fact and methods developed in \cite{BM3} we construct an exact, the so called, Alexander-Spanier  normal homology theory on the category $\mathcal{K}^2_{Top},$ which is isomorphic to the Steenrod homology theory on the subcategory of compact pairs $\mathcal{K}^2_{C}.$ Moreover, we give an axiomatic characterization of the constructed homology theory.  	
\end{abstract}


\maketitle

\section*{\bf Introduction} 
 On the category $\mathcal{K}^2_{CM}$ of pairs of compact metric spaces the exact homology theory  was defined by N. Steenrod \cite{St}, which is known as the classical Steenrod homology theory. J. Milnor \cite{Mil} constructed the exact homology theory on the category $\mathcal{K}^2_{C}$ of pairs of compact Hausdorff spaces, which is isomorphic to the Steenrod homology theory on the subcategory $\mathcal{K}^2_{CM}$ and which satisfies the so called  "modified continuity" property: if $ X_1 \leftarrow X_2 \leftarrow X_3 \leftarrow \dots $ is an inverse sequence of compact metric spaces with inverse limit $X$, then for each integer $n$ there is an exact sequence:
\begin{equation}\label{eq1}
0 \to {\varprojlim}^1 H_{n+1}(X_i) \xrightarrow{\text{{ $\beta$ }}} H_n(X) \xrightarrow{\text{$\gamma$}} \varprojlim H_n(X_i) \to 0,
\end{equation}
where $H_*$ is the Steenrod (Milnor) homology theory \cite{Mil}. There are exact homology theories defined by other authors \cite{Kol}, \cite{Cho} \cite{Sit}, \cite{BM}, \cite{In1}, \cite{EH1},\cite{EH2}, \cite{Mas}, \cite{Mdz3}, \cite{Skl}  which are isomorphic to the Steenrod homology theory on the category  $\mathcal{K}^2_{CM}$ and so, satisfy the modified continuity axiom. 

On the category $\mathcal{K}^2_{C}$ the axiomatic characterization is obtained by N. Berikashvili \cite{Ber}, L. Mdzinarishvili and Kh. Inasaridze \cite{InKh}, L. Mdzinarishvili \cite{Mdz1}, Kh. Inasaridze \cite{In2}. Consequently, in addition to the Eilenberg-Steenrod axioms one of the following axioms is required:

{\bf Universal Coefficient Formula}: For each pair $(X,A) \in \mathcal{K}^2_C$  and an abelian group $G$, there exists a  functorial exact sequence
\begin{equation}\label{eq12}
0 \os{}{\lra}
 \Ext(\check{H}^{n+1}(X,A);G) \os{}{\lra}  {H}_n(X,A;G) \os{}{\lra} \Hom(\check{H}^n(X,A);G) \lra 0,
\end{equation}
where $	\check{H}^{n+1}(-,-;G)$ is the  Alexandroff-\v{C}ech cohomology \cite{Ber}.

{\bf  Partial continuity:} Let $(X,A)$ be an inverse limit of an inverse system $\{(X_\alpha, A_\alpha), p_{\alpha, \alpha '} \}$ of compact polyhedra, then for each integer $n$ there is a functorial exact sequence:
\begin{equation}\label{eq3}
0 \to {\varprojlim}^1 H_{n+1}(X_\alpha, A_\alpha) \xrightarrow{\text{{ $\beta$ }}} H_n(X,A) \xrightarrow{\text{$\gamma$}} \varprojlim H_n(X_\alpha, A_\alpha) \to 0.
\end{equation}
\cite{InKh}, \cite{Mdz1}, \cite{Mrz1}, \cite{Mrz2}.

{\bf Continuity for an Injective Group:} Let $(X,A)$ be the inverse limit of an inverse system $\{(X_\alpha, A_\alpha), p_{\alpha, \alpha '} \}$ of compact polyhedra and $G$ be an injective abelian group, then for each integer $n$ there is an isomorphism:
\begin{equation}\label{eq4}
H_n(X,A;G) \approx \varprojlim H_n(X_\alpha,A_\beta ;G)
\end{equation}
\cite{In2}.

In the paper we will define the exact homology theory  $\bar{H}_*^N(-,-;G)$ using the method developed in \cite{BM1} and the Alexander-Spanier cochains based on the normal coverings. There is defined the continuity of an exact homology theory, which is the generalization of partial continuity. It is shown that any exact continuous theory is continuous for an injective abelian coefficient group and there is the Universal Coefficient Formula. Using the obtained properties we proved the uniqueness theorem.

 \section {Alexander-Spanier Normal Cohomology Theory}
 
  Let $X$ be a general topological space and $A$ be its subspace. Let us recall that a normal covering is defined as an open covering $\alpha=\left\{U_\alpha\right\}$, which admits a partition of unity $\left\{\varphi_{U_\alpha}|U_\alpha \in \alpha\right\}$ subordinated to $\alpha$, i.e. $\varphi_{U_\alpha} \ge 0$, $\sum \varphi_{U_\alpha}=1$ and for each $U_\alpha \in \alpha$ the support $supp \varphi_{U_\alpha}=\left\{x \in X |\varphi_{U_\alpha}(x)>0\right\}$ is contained in $U_\alpha$ \cite{MS}. $A$ is said to be $P$-embedded or normally embedded in $X$, if for each normal covering $\beta=\{V_\beta\}$ of $A$ there is a normal covering $\alpha=\{U_\alpha\}$ of $X$ such that $\alpha_{|A}=\{U_\alpha \cap A|U_\alpha \in \alpha\}$ is a refinement of the covering $\beta$. A pair $(X,A)$ of topological spaces is said to be closed $P$-pair if $A$ is closed and $P$-embedded in $X$.  Let $ \mathcal{K}^2_{Top}$ be the category of closed $P$-pairs \cite{MS}, \cite{Wat}.  Therefore, for paracompact spaces, any closed pair is  closed $P$-pair \cite{MS}. In this section, using the normal coverings we construct the Alexander-Spanier type cohomology theory, the so called, {\it Alexander-Spanier normal cohomology theory}. The methods of construction is the standard that is considered in \cite{Sp}, but we will recall and review some of them which are important for our purpose - to construct and axiomatically characterize an exact homology theory on the category $ \mathcal{K}^2_{Top}$. Note that, it is possible to define the Alexander-Spanier normal cohomology theory for each closed pair $(X,A)$, but it has "nice" properties for closed $P$-pairs.
  
  Let $X$ be a topological space, $G \in \mathcal{A}b$ be an abelian group. Consider the set $C^n(X;G)$ of all functions $\varphi:X^{n+1} \to G$,  where $X^{n+1}$ is $(n+1)$-fold product of the topological space $X$. If $\varphi_1, \varphi_2 \in C^{n+1}(X;G)$ and $(x_0,x_1, \dots , x_n) \in X^{n+1}$, then the addition in $C^n(X;G)$ can be defined by the formula:
 \begin{equation}\label{eq5}
 (\varphi_1+\varphi_2)(x_0,x_1, \dots ,x_n)=\varphi_1(x_0,x_1, \dots ,x_n)+\varphi_2(x_0,x_1, \dots ,x_n).
 \end{equation}
 It is clear that $C^n(X;G)$ is an abelian group by the given operation. The coboundary homomorphism $\delta:C^n(X;G)\to C^{n+1}(X;G)$ is defined by the formula:
\begin{equation}\label{eq6}
\delta(\varphi)(x_0,x_1, \dots ,x_{n+1})=\sum\limits_{i=0}^{n+1}(-1)^i\varphi (x_0, \dots , \hat{x}_i, \dots ,x_{n+1}),
\end{equation}  
where the coordinate with the hat $\hat{x}_i$ is omitted. In this case, it is known that $\delta \circ \delta =0$ and so, $C^*(X;G)=\left\{C^n(X;G), \delta\right\}$ is a cochain complex \cite{Sp}. 

Let $Cov_N(X)$ be the system af all normal coverings $\alpha=\left\{U_\alpha\right\}$ of $X$. An element $\varphi \in C^n(X;G)$ is said to be locally zero with respect to a normal covering (in short, $N$-locally zero), if there is a normal covering $\alpha=\{U_\alpha\}$ such that $\varphi$ vanishes on any $(n+1)$-tuple of $X$ which lies in some elements $U_\alpha $ of $\alpha$. Thus, if we define $\alpha^{n+1}=\bigcup\limits_{\alpha \in \mathcal{A}} U^{n+1}_{\alpha} \subset X^{n+1},$ then $\varphi$ vanishes on  $\alpha^{n+1}$. The subset of $C^n(X;G)$ consisting of all $N$-locally zero functions is a subgroup, denoted by $C^n_N(X;G)$. It is easy to check that, if $\varphi$ vanishes on $\alpha^{n+1}$, then $\delta (\varphi)$ vanishes on $\alpha^{n+2} \subset X^{n+2}$ and so, $C^*_N(X;G)=\left\{C^n_N(X;G), \delta\right\}$ is a cochain subcomplex of $C^*(X;G)$. Let $\bar{C}^*_N(X;G) \simeq {C}^*(X;G)/C^*_N(X;G)$ be the quotient cochain complex of $C^*(X;G)$ by $C^*_N(X;G)$. Let the $n$-dimensional cohomology group of the obtained cochain complex $\bar{C}^*_N(X;G)$ denote by $\bar{H}^*_N(X;G)$ and call {\it the Alexander-Spanier normal cohomology} of the topological space $X$ with coefficients in the group $G$. 

Let $f:X \to Y$ be a continuous map and $f^\# : {C}^*(Y;G) \to {C}^*(X;G)$ be the cochain map defined by the formula:
\begin{equation}\label{eq7}
f^\#(\varphi) (x_0,x_1,\dots ,x_n)=\varphi(f(x_0),f(x_1), \dots ,f(x_n)),~~~ \varphi \in C^n(Y;G), ~ x_0, x_1, \dots , x_n \in X.
\end{equation}
Note that if $\beta=\left\{V_\beta\right\}$  is a normal covering of $Y$, then $\alpha=f^{-1}(\beta)=\left\{f^{-1}(V_\beta)\right\}$ is a normal covering of $X$. If $\left\{\psi_{V_\beta}|V_\beta \in \beta \right\}$ is a partition of a unity subordinated to $\beta,$ then $\left\{\varphi_{U_\alpha}=\psi_{V_\beta}\circ f|U_\alpha=f^{-1}(V_\beta) \in \alpha \right\}$ is a partition of the unity subordinated to $\alpha$ \cite{MS}. On the other hand, if $\varphi \in C^*(Y;G)$ vanishes on $\beta^{n+1}$,  then $f^\#(\varphi)$ vanishes on $\alpha^{n+1}$ and so, $f^ \#$ maps $C^n_N(Y;G)$ into $C^n_N(X;G)$. Therefore, if $f:X \to Y$ is continuous, then it induces a cochain map
\begin{equation}\label{eq8}
f^\# : \bar{C}^*_N(Y;G) \to \bar{C}^*_N(X;G)
\end{equation}
and so, a homomorphism between the cohomology groups
\begin{equation}\label{eq9}
f^* : \bar{H}^*_N(Y;G) \to \bar{H}^*_N(X;G).
\end{equation}

Note that for each pair $(X,A) \in \mathcal{K}^2_{Top}$ the inclusion $i:A \to X$ induces an epimorphism $i^ \#: C^*(X;G) \to C^*(A;G),$ which maps $C^*_N(X;G)$ into $C^*_N(A;G)$. Consequently, it induces an epimorphism  
\begin{equation}\label{eq10}
i^\# : \bar{C}^*_N(X;G) \to \bar{C}^*_N(A;G).
\end{equation}
Let $\bar{C}^*_N(X,A;G)$ be the kernel of $i^\#$, then the cohomology groups of it is denoted by $\bar{H}^*_N(X,A;G)$ and will be said to be the Alexander-Spanier normal cohomology of pair $(X,A)$ with coefficients in group $G$.

For each closed $P$-pair  $(X,A) \in \mathcal{K}^2_{Top}$, denote the subcomplex of the cochain complex $C^*(X)$ consisting with functions $\varphi \in C^*(X)$, such that the restrictions on $A$ $\varphi_{|A}$ are $N$-locally zero by $C^*_N(X,A)$ . In this case, $C^*_N(X;G) \subset C^*_N(X,A;G)$ and $\bar{C}^*_N(X,A;G)\simeq C^*_N(X,A;G)/C^*_N(X;G)$. Note that this result is not true for any closed pair. It is essential that $A$ is closed and $P$-embedded to $X$. Our aim is to characterize the cochain complex $\bar{C}^*_N(X,A;G)$ as a direct limit. Consequently, we need the notion of a normal covering of a pair. 

A pair $(\alpha, \beta)$ is said to be a normal covering of a pair $(X,A)$ if $\alpha$ and $\beta$ are normal coverings of $X$ and $A$, respectively, and $\beta$ is a refinement of the restriction $\alpha_{|A}=\left\{U_\alpha \cap A| U_\alpha \in \alpha \right\}.$ A pair $(\alpha, \beta)$ is said to be  a refinement of $(\alpha', \beta')$ if $\alpha$ and $\beta$ are refinements of $\alpha'$ and $\beta'$, respectively. Let $Cov_N(X,A)$ be the direct set of all normal coverings $(\alpha, \beta)$ of pair $(X,A)$. Let for each $(\alpha, \beta) \in Cov_N(X,A)$, $X(\alpha)$ and $A(\beta)$ be abstract simplicial complexes (Vietorisian complexes) whose vertices are points of $X$ and $A$ and whose simplexes are fine subsets $F_X=\{x_1, x_2, \dots ,x_k\}$ of $X$ and  $F_A=\{a_1, a_2, \dots ,a_k\}$ of $A$ such that there is some $U_\alpha \in \alpha$ and $V_{\beta} \in \beta$ containing  $F_X \subset U_\alpha$ and $F_A \subset V_{\beta},$ respectively. By the definition  $A(\beta)$  is a subcomplex of $X(\alpha).$  Consider the corresponding cochain complex  $C^*(\alpha, \beta ;G).$ Note that an element $\varphi_\alpha \in C^n(\alpha, \beta ;G)$ is a function $\varphi_\alpha: \alpha ^{n+1} \to G$ which vanishes on ${\beta} ^{n+1} \subset \alpha ^{n+1}.$ Therefore, the system $Cov_N(X,A)$ induces a direct system $\left\{C^*(\alpha, \beta ;G)\right\}$ of the cochain complexes and so,  we have the limit cochain complex
\begin{equation}\label{eq11}
\varinjlim \left\{C^*(\alpha,\beta;G)\right\}.
\end{equation}
We need to show that this limit cochain complex is canonically isomorphic to $\bar{C}^*_N(X,A;G)$. If $\varphi \in C^*_N(X,A;G),$ then there is a normal covering $\gamma=\left\{W_\gamma\right\}$ of $A$ such that $\varphi$ vanishes on ${\gamma}^{n+1}.$ $A$ is closed and $P$-embedded in $X$ and so, there exists a normal covering $\alpha$ of $X$ such that $\beta=\alpha_{|A}=\left\{U_\alpha \cap A| U_\alpha  \in \alpha \right\}$ is a refinement of $\gamma$. Therefore, $(\alpha, \beta)$ is a normal covering of $(X,A)$ and the restriction $\varphi_{\alpha}:\alpha ^{n+1}\to G$ of the map $\varphi :X^{n+1}\to G$ vanishes on $\beta ^{n+1} \subset \gamma ^{n+1}$ and so, $\varphi_\alpha \in C^*(\alpha, \beta; G)$. Passing limit, we obtain the homomorphism
\begin{equation}\label{eq12}
\lambda:C^*(X,A;G)\to \varinjlim \left\{C^*(\alpha, \beta;G)\right\},
\end{equation}
which is the canonical cochain map.

\begin{theorem}\label{thm.1}
	The canonical cochain map
	\begin{equation}\label{eq13}
	\lambda:C^*_N(X,A;G)\to \varinjlim \left\{C^*(\alpha,\beta;G)\right\}
	\end{equation}
	is an epimorphism and has the kernel equal to $C^*_N(X;G)$.
\end{theorem}
\begin{proof} To prove that $\lambda$ is an epimorphism, let $\varphi_\alpha \in C^n(\alpha, \beta;G)$ and define $\varphi \in C^n(X;G)$ by the following formula:
	\begin{equation}\label{eq14}
	\varphi(x_0,x_1, \dots , x_n) =
	\begin{cases}
	\varphi_\alpha(x_0,x_1, \dots , x_n),~ if~ x_0,x_1, \dots , x_n \in U_\alpha,~ where~  U_\alpha \in \alpha\\
	0,~ otherwise.
	\end{cases}
	\end{equation}
The map $\varphi :X^{n+1} \to G$ vanishes on the $\beta ^{n+1} \subset A^{n+1}$ and so, $\varphi \in C^*_N(X,A;G)$. By definition $\varphi_\alpha=\varphi _{|\alpha^{n+1}}$ and so, $\lambda$ is an epimorphism. On the other hand, an element $\varphi \in C^n(X;G)$ is in the kernel of $\lambda$ if and only if there is some $\alpha$ such that $\lambda(\varphi)=\varphi_\alpha=\varphi _{|\alpha^{n+1}}=0$. Thus $\lambda(\varphi)=0$ if and only if  $\varphi \in C^n_N(X;G).$
\end{proof}

\begin{corollary}\label{Cor.1} For the cohomology theory $\bar{H}^*_N(X,A;G)$ there is a canonical isomorphism
\begin{equation}\label{eq15}
\bar{H}^*_N(X,A;G) \simeq \varinjlim \left\{{H}^n(C^*(\alpha, \beta;G)\right\},
\end{equation}
where $(\alpha, \beta) \in Cov_N(X, A).$
\end{corollary}

We need to see that the constructed cohomology is isomorphic to the Alexandroff-\v{C}ech cohomology based on all normal coverings \cite{Mor}, \cite{Bar}, \cite{Wat}, \cite{Ba1}, \cite{Ba2}. For this aim, consider any normal covering $(\alpha , \beta )$  of $(X,A) \in \mathcal{K}^2_{Top}$ and let $(R_1,R_2)$ be the pair of relations defined in the following way:
\begin{equation}\label{eq16}
xR_1U_\alpha \iff x \in U_\alpha , ~~~~~\forall x \in X, U_\alpha \in \alpha 
\end{equation}
\begin{equation}\label{eq17}
aR_2V_\beta \iff a \in V_\beta, ~~~~~\forall a \in A, V_\beta \in \beta. 
\end{equation}
In this case, the pair $(R_1,R_2)$ defines the pair of abstract simplicial complexes $\left(X(\alpha),A(\beta)\right)$ and $\left(N(\alpha),N(\beta)\right)$, where $\left(X(\alpha),A(\beta)\right)$ is the pair of Vietorisian complexes and $\left(N(\alpha),N(\beta)\right)$ is the pair of nerves \cite{Dow}. By Theorem 1 \cite{Dow} the homology groups  of the obtained complexes are isomorphic and so, by Corollary \ref{Cor.1} we obtain:
\begin{corollary}\label{Cor.2} For each closed $P$-pair $(X,A) \in \mathcal{K}^2_{Top}$ there is an isomorphism 
\begin{equation}\label{eq18}
\bar{H}^*_N(X,A;G) \simeq \check{H}^*_N(X,A;G),
\end{equation}
where $\check{H}^*_N(X,A;G)$ is the Alexandroff-\v{C}ech cohomology group based on all normal coverings.
\end{corollary}
By Lemma 2 (ii) \cite{Wat} the Alexandroff-\v{C}ech cohomology theory  based all normal coverings satisfies the relative homeomorphism axiom. Therefore, by Corollary \ref{Cor.2} we obtain:
\begin{corollary}\label{Cor.2.1}  If $f:\left(X,A\right) \to \left(Y,B\right)$ is a closed map of $P$-pairs such that f maps $X \setminus A$ homeomorphically onto $Y \setminus B$, then it induces an isomorphism:
	\begin{equation}\label{eq18.1}
	f_*:\bar{H}^*_N(Y,B;G) \os{\simeq}{\lra} \bar{H}^*_N(X,A;G),
	\end{equation}
\end{corollary} 
As it is known, the classical Alexandroff-\v{C}ech cohomology theory and consequently, the Alexander-Spanier cohomology theory are characterized on the category of compact pairs $\mathcal{K}^2_C$ by the Eilenberg-Spanier axioms and the continuity axiom \cite{Sp}. In particular, let $F:\mathcal{K}_C \to \mathcal{A}b$ be a contravariant functor. 

$CA$: $F$ is said to be continuous if each inverse limit  ${\bf p}:X \to {\bf X}=\left\{X_\alpha, p_{\alpha \beta }, \alpha \in \mathcal{A}\right\} $ induces direct limit $F({\bf p}): F({\bf X})=\left\{F(X_\alpha), F(p_{\alpha \beta }), \alpha \in \mathcal{A}\right\} \to F(X) .$ 

If for each integer $n \in \mathbb{Z}$ a cohomology group $H^n(-,-;G)$ is continuous, then $H^*(-,-;G)$ is said to be a continuous cohomology theory.  In the paper \cite{Wat},  is shown that for general topological spaces the classical Alexandroff-\v{C}ech cohomology theory does not  satisfy the continuity axiom (see Example 3 in \cite{Wat}). Using the resolution,  Watanabe defined the continuity axiom  and showed that the cohomology theory $\check{H}^*_N(X,A;G)$ satisfies it on the category $\mathcal{K}^2_{Top}$. In particular, let $\mathcal{K}$ be any subcategory of the category of topological spaces $\mathcal{K}_{Top}$ and $F:\mathcal{K} \to \mathcal{A}b$ be a contravariant functor. 

$CA^*$: $F$ is said to be continuous if each resolution  ${\bf p}:X \to {\bf X}=\left\{X_\alpha, p_{\alpha \beta }, \alpha \in \mathcal{A}\right\} $ induces direct limit $F({\bf p}): F({\bf X})=\left\{F(X_\alpha), F(p_{\alpha \beta }), \alpha \in \mathcal{A}\right\} \to F(X).$ 

On the category $\mathcal{K}_C$ of compact Hausdorff spaces ${\bf p}:X \to {\bf X}=\left\{X_\alpha, p_{\alpha \beta }, \alpha \in \mathcal{A}\right\} $ is a resolution iff ${\bf p}$ is an inverse limit (see Theorem 6.1.1,  \cite{MS}). Therefore, the  continuity axiom of the Watanabe sense is consistent with the classical definition. Using a resolution and the classical method of extension of (co)homology theory from the category $\mathcal{K}^2_{Pol_C}$ of pairs of compact polyhedra to the category $\mathcal{K}^2_C$ of compact pairs \cite{Mdz2}, the extension of a functor from the category of $\mathcal{K}^2_{Pol}$ of all pairs of polyhedra to the category $\mathcal{K}^2_{Top}$ of all $P$-pairs are defined \cite{Wat}. It is shown that the Alexandroff-\v{C}ech cohomology theory $\check{H}^*_N(X,A;G)$  based on all normal coverings on the category $\mathcal{K}^2_{Top}$ is the unique (up to a natural equivalence) continuous extension of the singular cohomology on polyhedral pairs see (see Theorem 7 of \cite{Wat}). On the other hand, the singular cohomology theory on the category of $\mathcal{K}^2_{Pol}$ is just an additive cohomology theory \cite{Mil}, \cite{BM1}. Consequently, as a corollary it is obtained that the cohomology theory $\check{H}^*_N(X,A;G)$ on the category $\mathcal{K}^2_{Top}$ is characterized by the Eilenberg-Steenrod Axioms, the continuity axiom and the additivity axiom (see Corollary 8 ii) in \cite{Wat}. Note that except the continuity axiom,  all other axioms are sufficient to be fulfilled only the subcategory $\mathcal{K}^2_{Pol}$ of all pairs of polyhedra. 

By Corollary 2 and Corollary 8 in \cite{Wat}, we have:
\begin{corollary}\label{Cor.2'} The Alexander-Spanier normal cohomolgy theory $\bar{H}^*_N(-,-;G)$ defined on the category $\mathcal{K}^2_{Top}$ of all closed $P$-pairs  is continuous in the Watanabe sense.
\end{corollary}

\section{Alexander-Spanier homology theory based on normal coverings}

In this section we use the method developed in the paper \cite{BM3} to construct an exact homology theory $\bar{H}^N_*(-,-;G)$, the so called {\it Alexander-Spanier normal homology}, such that for each closed $P$-pair there exists the Universal Coefficient Formula:
\begin{equation}\label{eq19}
0 \lra \Ext(\bar{H}^{n+1}_N(X,A);G) \os{}{\lra}  \bar{H}^N_n(X,A;G) \os{}{\lra} \Hom(\bar{H}^n_N(X,A);G) \lra 0,
\end{equation}
where $\bar{H}^{*}_N(-,-)$ is the Alexander-Spanier normal cohomology with the coefficient group  $\mathbb{Z}$. We will show that the constructed homology theory satisfies the Eilenberg-Steenrod axioms on the subcategory $\mathcal{K}^2_{Pol}$ of pairs of polyhedra and  by the uniqueness theorem on the category of paracompact  spaces given in \cite{San}, it is unique homology which is connected with the Alexandroff-\v{C}ech cohomology theory by the Universal Coefficient Formula.  On the other hand, we define continuity of an exact homology theory and show that the constructed homology theory is continuous in our sense.

Let $C^*$ be a cochain complex  and $0 \to G \os{\alpha}{\lra} G' \os{\beta}{\lra} G'' \to 0 $ be an injective resolution of a group $G$. Let $\beta_{\#}:\Hom(C^*;G') \to \Hom(C^*;G'')$ be the chain map induced by $\beta: G' \to G''$. Consider the cone  $C_*(\beta_{\#})=\left\{ C_n(\beta_{\#}), \partial \right\}=\left\{\Hom( C^*,\beta_{\#}), \partial \right\}$ of the chain map $\beta_{\#}$, i.e.
\begin{equation}\label{eq20}
C_n(\beta_{\#}) \simeq \Hom(C^{n};G') \oplus \Hom(C^{n+1};G''),
\end{equation}
\begin{equation}\label{eq21}
\partial(\varphi ',\varphi '')=(\varphi ' \circ \delta, \beta \circ \varphi ' -\varphi '' \circ \delta), ~~~ \forall (\varphi ',\varphi '') \in C_n(\beta_{\#}).
\end{equation}
If $f^\#:C^* \lra C'^*$ is a cochain map, then it induces the chain map $f_\#:C_* (\beta'_{\#}) \lra C_*(\beta_{\#})$. In particular, for each $n \in \mathbb{Z}$,  ${f}_n: C_n (\beta'_{\#}) \lra C_n (\beta_{\#})$ is defined by the formula
${f}_n(\varphi ', \varphi '')=(\varphi ' \circ f^n , \varphi '' \circ f^{n+1}).$  

\begin{lemma}\label{Lem.1} Each short exact sequence of cochain complexes
\begin{equation}\label{eq22}
0 \to C^* \os{f^\#}{\lra} C'^* \os{g^\#}{\lra} C''^* \to 0
\end{equation} 
induces a short exact sequence of chain complexes
\begin{equation}\label{eq23}
0 \to C_* (\beta''_{\#}) \os{g_\#}{\lra} C_* (\beta'_{\#}) \os{f_\#}{\lra} C_* (\beta_{\#}) \to 0.
\end{equation} 
\end{lemma}

\begin{proof} The groups $G'$ and $G''$ are injective abelian and so, the map $\beta : G' \to G''$ induces the following commutative diagram with the exact rows:
	\begin{equation}\label{eq24}
	\begin{tikzpicture}
	
	\node (A) {$0$};
	\node (B) [node distance=2.2cm, right of=A] {$\Hom(C''^*;G')$};
	c\node (C) [node distance=3.5cm, right of=B] {$ \Hom(C'^*;G')$};
	\node (D) [node distance=3.3cm, right of=C] {$\Hom(C^*;G')$};
	\node (E) [node distance=2cm, right of=D] {$0$};
	
	\node (A1) [node distance=1.5cm, below of=A] {$0$};
	\node (B1) [node distance=1.5cm, below of=B] {$\Hom(C''^*;G'')$};
	\node (C1) [node distance=1.5cm, below of=C] {$ \Hom(C'^*;G'')$};
	\node (D1) [node distance=1.5cm, below of=D] {$\Hom(C^*;G'')$};
	\node (E1) [node distance=1.5cm, below of=E] {$0$.};

	\draw[->] (A) to node [above]{}(B);
	\draw[->] (B) to node [above]{$g_\#$}(C);
	\draw[->] (C) to node [above]{$f_\#$}(D);
	\draw[->] (D) to node [above]{}(E);
	
	\draw[->] (A1) to node [above]{}(B1);
	\draw[->] (B1) to node [above]{$g_\#$}(C1);
	\draw[->] (C1) to node [above]{$f_\#$}(D1);
	\draw[->] (D1) to node [above]{}(E1);

	\draw[->] (B) to node [right]{$\beta''_\#$}(B1);
	\draw[->] (C) to node [right]{$\beta'_\#$}(C1);
	\draw[->] (D) to node [right]{$\beta_\#$}(D1);
	\end{tikzpicture}
	\end{equation}
Consequently, the diagram \eqref{eq24} induces the short exact sequence \eqref{eq23}. 
\end{proof}

Let $(X,A) \in \mathcal{K}^2_{Top}$ be a closed $P$-pair. Consider the short exact sequence of the Alexander-Spanier cochain complexes based on all normal coverings:
\begin{equation}\label{eq25}
0 \to \bar{C}^*_N(X,A;\mathbb{Z}) \os{j^\#}{\lra} \bar{C}^*_N(X;\mathbb{Z})  \os{i^\#}{\lra} \bar{C}^*_N(A;\mathbb{Z})  \to 0.
\end{equation} 
By Lemma \ref{Lem.1} the short exact sequence \eqref{eq25} induces a short exact sequence of chain complexes: 
\begin{equation}\label{eq26}
0 \to \bar{C}^N_* (A;G) \os{i_\#}{\lra} \bar{C}^N_* (X;G) \os{j_\#}{\lra} \bar{C}^N_* (X,A;G) \to 0,
\end{equation} 
where $\bar{C}^N_* (A;G)$, $\bar{C}^N_* (X;G)$ and $\bar{C}^N_* (X,A;G)$ are chain cones of the chain maps $\beta^A_\# : \Hom(\bar{C}^*_N(A;\mathbb{Z}) ;G') \to \Hom (\bar{C}^*_N(A;\mathbb{Z}) ;G'')$, $\beta^X_\# : \Hom(\bar{C}^*_N(X;\mathbb{Z}) ;G') \to \Hom (\bar{C}^*_N(X;\mathbb{Z}) ;G'')$ and $\beta^{(X,A)}_\# : \Hom(\bar{C}^*_N(X,A;\mathbb{Z}) ;G') \to \Hom (\bar{C}^*_N(X,A;\mathbb{Z}) ;G'')$, respectively. Let the homology groups of the chain complexes  $\bar{C}^N_* (A;G)$, $\bar{C}^N_* (X;G)$ and $\bar{C}^N_* (X,A;G)$ define by $\bar{H}^N_* (A;G)$, $\bar{H}^N_* (X;G)$ and $\bar{H}^N_* (X,A;G)$ and call {\it the Alexander-Spanier normal homology } of $A$, $X$ and $(X,A)$, respectively. By the sequence \eqref{eq26}, we obtain that it is an exact homology theory.
\begin{corollary}\label{Cor.3}{(Exactness Axiom)} For each $(X,A) \in \mathcal{K}^2_{Top}$ closed $P$-pair, there is a long exact homological sequence:
	 \begin{equation}\label{eq27}
	 \begin{tikzpicture}
	 
	 \node (A) {$\dots$};
	 \node (B) [node distance=2cm, right of=A] {$\bar{H}^N_n (A;G) $};
	 \node (C) [node distance=2.5cm, right of=B] {$\bar{H}^N_n (X;G) $};
	 \node (E) [node distance=5cm, right of=B] {$\bar{H}^N_n (X,A;G)$};
	 \node (F) [node distance=2.5cm, right of=E] {$\bar{H}^N_{n-1} (A;G)$};
	 \node (H) [node distance=2cm, right of=F] {$\dots$~~.};
	 
	 \draw[->] (A) to node [above]{$i_*$}(B);
	 \draw[->] (B) to node [above]{$i_*$}(C);
	 \draw[->] (C) to node [above]{$j_*$}(E);
	 \draw[->] (E) to node [above]{$E$}(F);
	 \draw[->] (F) to node [above]{$j_*$}(H);
	 
	 \end{tikzpicture}
	 \end{equation}
\end{corollary}
 
Let show that the constructed homology theory satisfies the homotopy axiom. 
\begin{lemma}\label{Lem.2} Each cochain homotopic maps $f^\#, g^\# :C^* \to C'^*$ induce  homotopic chain maps  $f_\#, g_\# :C_*(\beta '_\#) \to C_*(\beta_\#).$ 
\end{lemma}

\begin{proof} Let $D_\#=\left\{D_n\right\}$ be a cochain homotopy of cochain maps $f^\#, g^\# :C^* \to C'^*$. Therefore, for each $n \in \mathbb{Z}$ integer $D^n:C^n \to C'^{n-1}$ is a homomorphism such that
\begin{equation}\label{eq28}
\delta ' \circ D^n +D^{n+1} \circ \delta =f^n-g^n : C^n \to C'^n.
\end{equation}
Let $\bar{D}_\#=\left\{\bar{D}_n\right\}$ be the system of the homomorphisms $\bar{D}_n:C_n(\beta'_\#) \to C_{n+1}(\beta_\#)$ defined by the formula:
\begin{equation}\label{eq29}
\bar{D}_n(\varphi',\varphi '')=(\varphi ' \circ D^n , -\varphi '' \circ D^{n+1}),~~~~\forall ~ (\varphi', \varphi'') \in C_n(\beta'_\#).
\end{equation}
Let show that $\bar{D}_\#=\left\{\bar{D}_n\right\}$ is degree $1$ chain map such that
\begin{equation}\label{eq30}
\bar{D}_n \circ \partial' + \partial \circ \bar{D}_{n+1}=f_n-g_n:C_*(\beta'_\#)\to C_*(\beta_\#).
\end{equation}
Indeed, we have
$$\left(\bar{D}_{n} \circ \partial' \right)\left(\varphi ' , \varphi ''\right)=\bar{D}_{n}\left(\varphi ' \circ \delta' , \beta \circ \varphi '  -\varphi '' \circ \delta' \right)=$$
\begin{equation}\label{eq31}
\left(\varphi ' \circ \delta' \circ D^n ,-\left( \beta \circ \varphi '  -\varphi '' \circ \delta' \right) \circ D^{n+1}  \right)=\left(\varphi ' \circ \delta' \circ D^n ,-\beta \circ \varphi ' \circ D^{n+1}  + \varphi '' \circ \delta'  \circ D^{n+1}  \right).
\end{equation}
$$\left(\partial \circ \bar{D}_{n+1}\right)\left(\varphi ' , \varphi ''\right)=\partial\left(\bar{D}_{n+1}\left(\varphi ' , \varphi ''\right)\right)=\partial\left(\varphi ' \circ D^{n+1} , -\varphi '' \circ D^{n+2} \right)=$$
\begin{equation}\label{eq32}
\left(\varphi ' \circ D^{n+1} \circ \delta, \beta \circ \varphi ' \circ D^{n+1} -\left(-\varphi '' \circ D^{n+2} \circ \delta \right) \right)=\left(\varphi ' \circ D^{n+1} \circ \delta, \beta \circ \varphi ' \circ D^{n+1}+\varphi '' \circ D^{n+2} \circ \delta \right).
\end{equation}
Hence, we have
$$\left(\bar{D}_n \circ \partial' + \partial \circ \bar{D}_{n+1}\right) \left( \varphi ', \varphi ''\right)=  \left(\varphi ' \circ \delta' \circ D^n ,-\beta \circ \varphi ' \circ D^{n+1}  + \varphi '' \circ \delta'  \circ D^{n+1}  \right)+\left(\varphi ' \circ D^{n+1} \circ \delta, \beta \circ \varphi ' \circ D^{n+1}+\varphi '' \circ D^{n+2} \circ \delta \right)=$$
$$\left(\varphi ' \circ \delta' \circ D^n +\varphi ' \circ D^{n+1} \circ \delta,-\beta \circ \varphi ' \circ D^{n+1}  + \varphi '' \circ \delta'  \circ D^{n+1} +\beta \circ \varphi ' \circ D^{n+1}+\varphi '' \circ D^{n+2} \circ \delta   \right)=$$
$$\left(\varphi ' \circ \left(\delta' \circ D^n + D^{n+1} \circ \delta   \right), \varphi '' \circ \left( \delta'  \circ D^{n+1}+ D^{n+2} \circ \delta  \right)  \right)=\left(\varphi ' \circ \left(f^n-g^n \right), \varphi '' \circ \left( f^{n+1}-g^{n+1} \right)  \right)=$$
\begin{equation}\label{eq33}
\left(\varphi ' \circ f^n, \varphi '' \circ f^{n+1}   \right)-\left(\varphi ' \circ g^n, \varphi '' \circ g^{n+1}   \right)=f_n\left( \varphi ', \varphi ''\right)-g_n\left( \varphi ', \varphi ''\right).
\end{equation}
\end{proof}

\begin{corollary}\label{Cor.4}{(Homotopy Axiom)} If $f,g:(X,A) \to (X,B)$ are homotopic continuous maps of closed $P$-pairs, then
	\begin{equation}\label{eq34}
	f_*=g_*:\bar{H}^N_*(X,A) \to \bar{H}^N_*(Y,B).
	\end{equation}
\end{corollary}

By Theorem 1 of \cite{BM3}, for the homology theory $\bar{H}^N_*(-,-;G)$ there is the Universal Coefficient Formula.

\begin{corollary}\label{Cor.5}(Universal Coefficient Formula) For each $(X,A) \in \mathcal{K}^2_{Top}$ closed $P$-pair and an abelian group $G$, there exists a short exact sequence
\begin{equation}\label{eq38}
0 \lra \Ext(\bar{H}^{n+1}_N(X,A);G) \os{}{\lra}  \bar{H}^N_n(X,A;G) \os{}{\lra} \Hom(\bar{H}^n_N(X,A);G) \lra 0,
\end{equation}
where $\bar{H}^{n+1}_N(-,-)$ is the Alexander-Spanier normal cohomology with the coefficient group $\mathbb{Z}$.
\end{corollary}
Using the Universal Coefficient Formula we simply obtain the relative homeomorphism and the dimension axioms.
	\begin{theorem}\label{thm.2} (Relative Homeomorphism Axiom) The Alexander-Spannier normal homology theory $H^N_*(-,-;G)$ is invariant under relative homeomorphism.
	\end{theorem}
	\begin{proof} Consider the diagram induced by $f:(X,A) \to (Y,B)$:
		\begin{equation}\label{eq95}
		\begin{tikzpicture}
		
		\node (A) {$0$};
		\node (B) [node distance=2.5cm, right of=A] {$\Ext \left(\bar{H}^{n+1}_N \left(Y,B;\mathbb{Z}\right);G\right) $};
		\node (C) [node distance=3.5cm, right of=B] {${H}_n\left(Y,B;G\right)$};
		\node (D) [node distance=3.5cm, right of=C] {$\Hom \left(\bar{H}^n_N\left(Y,B;\mathbb{Z}\right);G \right)$};
		\node (E) [node distance=2.7cm, right of=D] {$0~,$};

		\draw[->] (A) to node [above]{}(B);
		\draw[->] (B) to node [above]{}(C);
		\draw[->] (C) to node [above]{}(D);
		\draw[->] (D) to node [above]{}(E);
		
		\node (A1) [node distance=1.5cm, above of=A] {0};
		\node (B1) [node distance=1.5cm, above of=B] {$\Ext \left(\bar{H}^{n+1}_N \left(X,A;\mathbb{Z}\right);G\right) $};
		\node (C1) [node distance=1.5cm, above of=C] {${H}_n\left(X,A;G\right)$};
		\node (D1) [node distance=1.5cm, above of=D] {$\Hom \left(\bar{H}^n_N\left(X,A;\mathbb{Z} \right);G\right)$};
		\node (E1) [node distance=1.5cm, above of=E] {0};

		\draw[->] (A1) to node [above]{}(B1);
		\draw[->] (B1) to node [above]{}(C1);
		\draw[->] (C1) to node [above]{}(D1);
		\draw[->] (D1) to node [above]{}(E1);

		\draw[<-] (B) to node [right]{$\Ext\left(f^*;G\right)$}(B1);
		\draw[<-] (C) to node [right]{$f_*$}(C1);
		\draw[<-] (D) to node [right]{$\Hom\left(f^*;G\right)$}(D1);
		\end{tikzpicture}
		\end{equation}
		By Corollary \ref{Cor.2.1} the homomorphisms $f^*:\bar{H}^{p}_N \left(Y,B;\mathbb{Z}\right) \to \bar{H}^{p}_N \left(Y,B;\mathbb{Z}\right)$ is an isomorphism for all $p \ge 0$ and therefore, $f_*:\bar{H}_{n}^N \left(X,A;\mathbb{Z}\right) \to \bar{H}_{n}^N \left(Y,B;\mathbb{Z}\right)$ is an isomorphism.
	\end{proof}
\begin{theorem}\label{thm.3} {(Dimension Axiom)} For each one point topological space $P$, there is an isomorophism 
	\begin{equation}\label{eq39}
	\bar{H}^N_n(P;G)=\begin{cases}
	G, ~~if~~ n=0\\
	0, ~~if~~ n\ne 0.
	\end{cases}
	\end{equation}
\end{theorem}
\begin{proof} Consider the corresponding Universal Coefficient Formula:
	\begin{equation}\label{eq40}
	0 \lra \Ext(\bar{H}^{n+1}_N(P;\mathbb{Z});G) \os{}{\lra}  \bar{H}^N_n(P;G) \os{}{\lra} \Hom(\bar{H}^n_N(P;\mathbb{Z});G) \lra 0.
	\end{equation}
The Alexander-Spanier normal cohomology theory $\bar{H}^*_N(-,-)$ satisfies the dimension axiom.  Therefore, we have
    \begin{equation}\label{eq41}
	\bar{H}^n_N(P;\mathbb{Z})=\begin{cases}
	\mathbb{Z}, ~~if~~ n=0\\
	0, ~~if~~ n\ne 0.
	\end{cases}
	\end{equation}
Thus, 	$\Ext(\bar{H}^{n+1}_N(P;\mathbb{Z});G)\simeq 0$ and so, we have the isomorphism
\begin{equation}\label{eq42}
  \bar{H}^N_n(P;G) \simeq  \Hom(\bar{H}^n_N(P;\mathbb{Z});G)=\begin{cases}
  \Hom(\mathbb{Z};G)\simeq G, ~~if~~ n=0\\
  \Hom(0;G)\simeq 0, ~~if~~ n\ne 0.
  \end{cases}.
\end{equation}
\end{proof}

\section{Some properties of Alexander-Spanier homology theory on the category $\mathcal{K}^2_{Top}$}

The continuity axiom for an exact homology theory $H_*(-,-;G)$ on the category of compact spaces is defined in the paper \cite{Mdz1} (see Definition 1). We can formulate it in the following way:

$C$: An exact homology theory $H_*(-;G):\mathcal{K}_C \to \mathcal{A}b$ is said to be continuous if each inverse limit  ${\bf p}:X \to {\bf X}=\left\{X_\alpha, p_{\alpha \beta }, \alpha \in \mathcal{A}\right\} $ induces a long exact sequence:
\begin{equation}\label{eq43}
\begin{tikzpicture}\small

\node (A) {$\dots$};
\node (B) [node distance=2cm, right of=A] {$ {\varprojlim} ^ {(3)}   {H} _{n+2}(X_\alpha;G) $};
\node (C) [node distance=3cm, right of=B] {${\varprojlim} ^ {(1)}  H_{n+1}(X_\alpha;G)$};
\node (D) [node distance=2.5cm, right of=C] {$ H_n(X;G)$};
\node (E) [node distance=2.2cm, right of=D] {${\varprojlim}  {H} _{n}(X_\alpha;G) $};
\node (F) [node distance=2.8cm, right of=E] {${\varprojlim} ^ {(2)}  {H}_ {n+1}(X_\alpha;G) $};
\node (H) [node distance=2.2cm, right of=F] {$\dots$~.};

\draw[->] (A) to node [above]{}(B);
\draw[->] (B) to node [above]{}(C);
\draw[->] (C) to node [above]{}(D);
\draw[->] (D) to node [above]{}(E);
\draw[->] (E) to node [above]{}(F);
\draw[->] (F) to node [above]{}(H);

\end{tikzpicture}
\end{equation}

Using the continuity axiom, an exact homology theory $H_*(-,-;G)$ is characterized on the category $\mathcal{K}^2_C$ (see Corollaries 3, 4, 5 in \cite{Mdz1}).Here we formulate three different properties of an exact homology theory on the category $\mathcal{K}^2_{Top},$ which are the generalization of the axioms proposed by N. Berikashvili \cite{Ber},  L. Mdzinarishvili and Kh. Inasaridze \cite{InKh}, L. Mdzinarishvili \cite{Mdz2} and Kh. Inasaridze \cite{In1} on the category $\mathcal{K}^2_C$. Note that for paracompact spaces S. Saneblidze \cite{San} generalized the result obtained by N. Berikashvili \cite{Ber} for compact spaces.   

Let $H_*(-,-;G)$ be a homological functor defined on the category $\mathcal{K}^2_{Top}$ of closed $P$-pairs.   

$CEH$ ({\it Continuity for an Exact Homology}): For each resolution  ${\bf p}:(X,A) \to {\bf (X,A)}=\left\{(X_\alpha, A_\beta), p_{\alpha \beta }, \alpha \in \mathcal{A}\right\} $ of closed $P$-pair $(X,A) \in \mathcal{K}^2_{Top}$ and an abelian
group $G$, there exists a functorial long exact sequence:
\begin{equation}\label{eq44}
\begin{tikzpicture}\small

\node (A) {$\dots$};
\node (B) [node distance=2.2cm, right of=A] {$ {\varprojlim} ^ {(3)}   {H} _{n+2}(X_\alpha, A_\beta;G) $};
\node (C) [node distance=3.2cm, right of=B] {${\varprojlim} ^ {(1)}  H_{n+1}(X_\alpha,A_\beta;G)$};
\node (D) [node distance=2.6cm, right of=C] {$ H_n(X,A;G)$};
\node (E) [node distance=2.3cm, right of=D] {${\varprojlim}  {H} _{n}(X_\alpha,A_\beta;G) $};
\node (F) [node distance=3cm, right of=E] {${\varprojlim} ^ {(2)}  {H}_ {n+1}(X_\alpha,A_\beta;G) $};
\node (H) [node distance=2.2cm, right of=F] {$\dots$~.};

\draw[->] (A) to node [above]{}(B);
\draw[->] (B) to node [above]{}(C);
\draw[->] (C) to node [above]{}(D);
\draw[->] (D) to node [above]{}(E);
\draw[->] (E) to node [above]{}(F);
\draw[->] (F) to node [above]{}(H);

\end{tikzpicture}
\end{equation}

$CIG$ ({\it Continuity for an Injective Group}): For each resolution  ${\bf p}:(X,A) \to {\bf (X,A)}=\left\{(X_\alpha, A_\beta), p_{\alpha \beta }, \alpha \in \mathcal{A}\right\} $ of closed $P$-pair $(X,A) \in \mathcal{K}^2_{Top}$ and  an injective abelian group $G$, there exists an isomorphism
\begin{equation}\label{eq45}
H_n(X,A;G) \approx \varprojlim H_n(X_\alpha,A_\beta ;G).
\end{equation}

$UCF$ ({\it Universal Coefficient Formula}): For each $(X,A) \in \mathcal{K}^2_{Top}$ closed $P$-pair and an abelian group $G$, there exists a functorial exact sequence:
\begin{equation}\label{eq46}
0 \lra \Ext\left(\bar{H}^{n+1}_N(X,A);G\right) \os{}{\lra}  {H}_n\left(X,A;G\right) \os{}{\lra} \Hom(\bar{H}^n_N\left(X,A);G\right) \lra 0,
\end{equation}
where $	\bar{H}^{n+1}_N(-,-;G)$ is the Alexander-Spanier normal cohomology.

Note that if the property $UCF$  is satisfied only for some class of pairs $M^2 \subset  \mathcal{K}^2_{Top}$, then we will use the notation $UCF_{M^2}$.

\begin{definition}[{see \cite{BM3} }]\label{Def.1}
	A direct system $\mathbf{C}^*=\{C^*_\alpha \}$ of  cochain complexes $C^*_\alpha$ is said to be associated with a cochain complex $C^*,$ if there is a homomorphism $\mathbf{C}^* \to C^*$ such that for each $n \in \mathbb{Z}$ the induced homomorphism
	\begin{equation}\label{eq47}
	\varinjlim H^*(C^*_\alpha) \to H^*(C^*)
	\end{equation}
	is an isomorphism.
\end{definition}

By Lemma 5.5 \cite{BM3}, it is known that if a direct system $\mathbf{C}^*=\{C^*_\alpha \}$ of cochain complexes $C^*_\alpha$ is associated with a cochain complex $C^*$, then there is an infinite exact sequence: 
	$$\cdots \lra \llm^{(2k+1)} \bar{H}_{n+k+1}(C^*_\alpha;G) \lra \cdots \lra \llm^{(3)} \bar{H}_{n+2}(C^*_\alpha;G) \lra \llm^{(1)} \bar{H}_{n+1}(C^*_\alpha;G) \lra$$
	\begin{equation} \label{eq48}
	\lra \bar{H}_{n}(C^*;G) \os{\pi_*}{\lra} \llm \bar{H}_{n}(C^*_\alpha;G) \lra  \llm^{(2)} \bar{H}_{n+1}(C^*_\alpha;G) \lra \cdots \lra \llm^{(2k)} \bar{H}_{n+k}(C^*_\alpha;G) \lra \cdots\, .
	\end{equation}
	where $\bar{H}_{*}(C^*;G)=H_*\left(\Hom\left(C^*;\beta_\#\right)\right)=H_*\left(C_*\left(\beta_\#\right)\right)$ and $\bar{H}_{*}(C^*_\alpha;G)=H_*\left(\Hom(C^*_\alpha;\beta_\#)\right)=H_*\left(C^\alpha_*(\beta_\#)\right).$

\begin{theorem}\label{thm.4} The Alexander-Spanier normal homology theory $\bar{H}^N_*(-;G)$ satisfies CEH property. 
\end{theorem}

\begin{proof} Let ${\bf p}:X \to {\bf X}=\left\{X_\alpha, p_{\alpha \beta }, \alpha \in \mathcal{A}\right\} $  be a resolution. Assume that it is a polyhedral resolution. Consider the direct system $\mathbf{C}^*_N=\{\bar{C}^*_N(X_\alpha) \}$ of the cochain complexes, where  $\bar{C}^*_N(X_\alpha)$ is the Alexander-Spanier cochain complex based on the normal coverings. In this case, ${\bf p}:X \to {\bf X}=\left\{X_\alpha, p_{\alpha \beta }, \alpha \in \mathcal{A}\right\} $ induces the homomorphism
	\begin{equation}\label{eq49}
	{\bf p}^\#: \varinjlim \bar{C}^*_N(X_\alpha;G) \to \bar{C}^*_N(X;G),
	\end{equation}
which itself induces the homomorphism
	\begin{equation}\label{eq50}
    {\bf p}^*: \varinjlim \bar{H}^*_N(X_\alpha;G) \to \bar{H}^*_N(X;G).
    \end{equation}	
By Corollary \ref{Cor.2} we have $\bar{H}^*_N(X;G) \simeq \check{H}^*_N(X;G)$. On the other hand, for each $\alpha$, the space $X_\alpha$ is a polyhedron and so the cohomology group $\check{H}^*_N(X_\alpha;G)$ is isomorphic to the classical Alexander-Spanier cohomology group $\check{H}^*(X_\alpha;G)$, which itself is isomorphic to the singular cohomology group  $\check{H}^*_s(X_\alpha;G)$.(see Corollary 6.9.7 \cite{Sp}. Therefore, by Corollary 8 ii) in \cite{Wat} \eqref{eq50} is an isomorphism. Consequently, by Lemma 5.5 \cite{BM3}, we have the following long exact sequence:
\begin{equation}\label{eq51}
\begin{tikzpicture}

\node (A) {$\dots$};
\node (B) [node distance=2.2cm, right of=A] {$ {\varprojlim} ^ {(3)}   \bar{H}^N _{n+2}(X_\alpha;G) $};
\node (C) [node distance=3.2cm, right of=B] {${\varprojlim} ^ {(1)}  \bar{H}^N_{n+1}(X_\alpha;G)$};
\node (D) [node distance=2.6cm, right of=C] {$ \bar{H}^N_n(X;G)$};
\node (E) [node distance=2.3cm, right of=D] {${\varprojlim}  \bar{H}^N _{n}(X_\alpha;G) $};
\node (F) [node distance=2.8cm, right of=E] {${\varprojlim} ^ {(2)}  \bar{H}^N_ {n+1}(X_\alpha;G) $};
\node (H) [node distance=2.2cm, right of=F] {$\dots$~.};

\draw[->] (A) to node [above]{}(B);
\draw[->] (B) to node [above]{}(C);
\draw[->] (C) to node [above]{}(D);
\draw[->] (D) to node [above]{}(E);
\draw[->] (E) to node [above]{}(F);
\draw[->] (F) to node [above]{}(H);

\end{tikzpicture}
\end{equation}
\end{proof}
By Theorem of \cite{San}, the Alexander-Spanier homology $\bar{H}^N_*(-,-;G)$ is unique (up to a natural equivalence) on the category $\mathcal{K}^2_{Pol}$ and therefore, by Theorem \ref{thm.4}, we obtain (cf. \cite{Ku}, \cite{BK}):

\begin{corollary}\label{Cor.6} If $H_*(-,-;G)$ and $H'_*(-,-;G)$ are exact homology theories on a subcategory $\mathcal{K}^2 \subset \mathcal{K}^2_{Top},$ which  satisfies $UCF_{\mathcal{K}^2_{Pol}}$ and $CEH$ properties, then any natural  transformation (a mapping between corresponding long exact sequences involving derivative limits) $T:H_*\to H'_*$, which is an isomorphism for the one point topological space, is the isomorphism for any closed $P$-pair $(X,A) \in \mathcal{K}^2$. 
\end{corollary}

In cases when homologies $\bar{H}_*$ and $\bar{H}'_*$ are generated by cochain complexes $C^*$ and $C'^*,$  using the method developed in \cite{BM3}, any cochain map $f^*:C'^* \to C^*$ induces a natural transformation $f_*:\bar{H}_*\to \bar{H}'_*$ that is mentioned in Corollary \ref{Cor.6}:

\begin{equation}\label{eq52}
\begin{tikzpicture}\small

\node (A) {$\dots$};
\node (B) [node distance=2.2cm, right of=A] {$ {\varprojlim} ^ {(3)}   \bar{H}^N _{n+2}(X_\alpha, A_\beta;G) $};
\node (C) [node distance=3.2cm, right of=B] {${\varprojlim} ^ {(1)}  \bar{H}^N_{n+1}(X_\alpha,A_\beta;G)$};
\node (D) [node distance=2.6cm, right of=C] {$ \bar{H}^N_n(X,A;G)$};
\node (E) [node distance=2.3cm, right of=D] {${\varprojlim}  \bar{H}^N _{n}(X_\alpha,A_\beta;G) $};
\node (F) [node distance=2.8cm, right of=E] {${\varprojlim} ^ {(2)}  \bar{H}^N_ {n+1}(X_\alpha,A_\beta;G) $};
\node (H) [node distance=2.2cm, right of=F] {$\dots$~};

\draw[->] (A) to node [above]{}(B);
\draw[->] (B) to node [above]{}(C);
\draw[->] (C) to node [above]{}(D);
\draw[->] (D) to node [above]{}(E);
\draw[->] (E) to node [above]{}(F);
\draw[->] (F) to node [above]{}(H);

\node (A) {$\dots$};
\node (A1) [node distance=1.5cm, below of=A] {$\dots$};
\node (B1) [node distance=1.5cm, below of=B] {$ {\varprojlim} ^ {(3)}   \bar{H}'^N _{n+2}(X_\alpha, A_\beta;G) $};
\node (C1) [node distance=1.5cm, below of=C] {${\varprojlim} ^ {(1)}  \bar{H}'^N_{n+1}(X_\alpha,A_\beta;G)$};
\node (D1) [node distance=1.5cm, below of=D] {$ \bar{H}'^N_n(X,A;G)$};
\node (E1) [node distance=1.5cm, below of=E] {${\varprojlim}  \bar{H}'^N _{n}(X_\alpha,A_\beta;G) $};
\node (F1) [node distance=1.5cm, below of=F] {${\varprojlim} ^ {(2)}  \bar{H}'^N_ {n+1}(X_\alpha,A_\beta;G) $};
\node (H1) [node distance=1.5cm, below of=H] {$\dots$~.};

\draw[->] (A1) to node [above]{}(B1);
\draw[->] (B1) to node [above]{}(C1);
\draw[->] (C1) to node [above]{}(D1);
\draw[->] (D1) to node [above]{}(E1);
\draw[->] (E1) to node [above]{}(F1);
\draw[->] (F1) to node [above]{}(H1);

\draw[->] (B) to node [right]{$f_*$}(B1);
\draw[->] (C) to node [right]{$f_*$}(C1);
\draw[->] (D) to node [right]{$f_*$}(D1);
\draw[->] (E) to node [right]{$f_*$}(E1);
\draw[->] (F) to node [right]{$f_*$}(F1);

\end{tikzpicture}
\end{equation}
For example, if $C^*_s(X;G)$ is the singular cochain complex of topological spaces $X$ and  $\bar{C}^s_*(X;G)=\Hom\left(C^*_s(X);\beta_\#\right),$ then we can construct   the homology $\bar{H}_*^s(X;G)$  of the obtained chain complex $\bar{C}^s_*(X;G).$ Therefore, $\bar{H}_*^s(-;G)$ is the homology generated by the singular cochain complex  ${C}^*_s(-;G).$ It is known that there is a homomorphism $j^\#:\bar{C}^*(X;G) \to C^*_s(X;G)$ from the Alexander-Sapnier cochain complex to the singular cochain complex, which induces the isomorphism $j^*:\bar{H}^*(X;G) \to \bar{H}^*_s(X;G)$ on the category of manifolds. On the other hand, for a manifold $X$ the Alexander-Spanier cochain complex $\bar{C}^*(X;G)$ coincides with the Alexander-Spanier normal cochain complex $\bar{C}_N^*(X;G).$ Therefore, by Corollary \ref{Cor.6}, there is an isomorphism:
\begin{equation}\label{eq53}
j_*:\bar{H}^s_*(X;G) \os{\simeq}{\lra} \bar{H}^N_*(X;G).
\end{equation}

Now, using the method developed in \cite{BM2} we find the relation between $UCF$, $CEH$ and $CIG$ axioms.

\begin{theorem}\label{thm.5} 
	If ${H}_*$ is an exact homological functor defined on the category $\mathcal{K}^2_{Top}$ of closed $P$-pairs, which  satisfies $UCF_{\mathcal{K}^2_{Pol}}$ and $CEH$ properties, then it satisfies $CIG$ property for all closed $P$-pairs.
\end{theorem} 

\begin{proof} 
	Let ${\bf p}:(X,A) \to {\bf (X,A)}=\left\{(X_\alpha, A_\beta), p_{\alpha \beta }, \alpha \in \mathcal{A}\right\} $ be a resolution. The pair $(X,A)$ is a closed $P$-pair and so, the restriction ${\bf p}_{|A}:A \to {\bf A}=\left\{ A_\beta, p_{\alpha \alpha' |{A_\alpha}}, \alpha \in \mathcal{A}\right\} $ is a resolution \cite{MS}. Therefore, it is sufficient to prove in the absolute case. Hence, for each resolution ${\bf p}:X \to {\bf X}=\left\{X_\alpha, p_{\alpha \alpha' }, \alpha \in \mathcal{A}\right\} $ and  an injective abelian group $G'$ we have to show that:
	\begin{equation}\label{eq54}
    H_n(X;G') \simeq {\varprojlim} H_{n}(X_\alpha; G').
    \end{equation}
    By the condition of the Theorem, for each  abelian group $G$ we have the following long exact sequence:
	\begin{equation}\label{eq55}
	\begin{tikzpicture}
	
	\node (A) {$\dots$};
	\node (B) [node distance=2cm, right of=A] {$ {\varprojlim} ^ {(3)}   {H} _{n+2}(X_\alpha;G) $};
	\node (C) [node distance=3cm, right of=B] {${\varprojlim} ^ {(1)}  H_{n+1}(X_\alpha;G)$};
	\node (D) [node distance=2.5cm, right of=C] {$ H_n(X;G)$};
	\node (E) [node distance=2.2cm, right of=D] {${\varprojlim}  {H} _{n}(X_\alpha;G) $};
	\node (F) [node distance=2.8cm, right of=E] {${\varprojlim} ^ {(2)}  {H}_ {n+1}(X_\alpha;G) $};
	\node (H) [node distance=2.2cm, right of=F] {$\dots$~.};
	
	\draw[->] (A) to node [above]{}(B);
	\draw[->] (B) to node [above]{}(C);
	\draw[->] (C) to node [above]{}(D);
	\draw[->] (D) to node [above]{}(E);
	\draw[->] (E) to node [above]{}(F);
	\draw[->] (F) to node [above]{}(H);
	
	\end{tikzpicture}
	\end{equation}
	Therefore, we should show that for each  injective abelian group $G$ the derivatives are trivial:
	\begin{equation}\label{eq56}
	{\varprojlim}^{(i)} H_{n+1}(X_\alpha; G)=0, ~~~~ i \ge 1.
	\end{equation}
	Indeed, for each polyhedron  $X_\alpha$ we have the sequnce:
	\begin{equation}\label{eq57}
	0 \to \Ext\left(\bar{H}^{n+1}_N(X_{\alpha});G\right) \to H_n (X_{\alpha};G) \to \Hom\left(\bar{H}^n_N(X_{\alpha});G\right) \to 0,
	\end{equation}
	which induces the long exact sequence:

	\begin{equation}\label{eq58}
	\begin{tikzpicture}
	\node (A) {0};
	\node (B) [node distance=3cm, right of=A] {$\varprojlim \Ext\left(\bar{H}^{n+1}_N(X_{\alpha}),G\right)$};
	\node (C) [node distance=4cm, right of=B] {$\varprojlim H_n (X_{\alpha};G)$};
	\node (D) [node distance=4cm, right of=C] {$\varprojlim \Hom\left(\bar{H}^n_N(X_{\alpha}),G\right)$};
	\node (E) [node distance=3cm, right of=D] {};
	
	\draw[->] (A) to node [above]{}(B);
	\draw[->] (B) to node [above]{}(C);
	\draw[->] (C) to node [above]{}(D);
	\draw[->] (D) to node [above]{}(E);
	
	\node (A1) [node distance=1cm, below of=A] {};
	\node (B1) [node distance=1cm, below of=B] {${\varprojlim}^1 \Ext\left(\bar{H}^{n+1}_N(X_{\alpha}),G\right)$};
	\node (C1) [node distance=1cm, below of=C] {${\varprojlim}^1 H_n (X_{\alpha};G_0)$};
	\node (D1) [node distance=1cm, below of=D] {${\varprojlim}^1 \Hom\left(\bar{H}^n_N(X_{\alpha}),G_0\right)$};
	\node (E1) [node distance=1cm, below of=E] {};
	
	\draw[->] (A1) to node [above]{}(B1);
	\draw[->] (B1) to node [above]{}(C1);
	\draw[->] (C1) to node [above]{}(D1);
	\draw[->] (D1) to node [above]{}(E1);
	
	\node (A2) [node distance=1cm, below of=A1] {};
	\node (B2) [node distance=1cm, below of=B1] {${\varprojlim}^2 \Ext\left(\bar{H}^{n+1}_N(X_{\alpha}),G\right)$};
	\node (C2) [node distance=1cm, below of=C1] {${\varprojlim}^2 H_n (X_{\alpha};G_0)$};
	\node (D2) [node distance=1cm, below of=D1] {${\varprojlim}^2 \Hom\left(\bar{H}^n_N(X_{\alpha}),G_0\right)$};
	\node (E2) [node distance=1cm, below of=E1] {$\dots ~.$};
	
	\draw[->] (A2) to node [above]{}(B2);
	\draw[->] (B2) to node [above]{}(C2);
	\draw[->] (C2) to node [above]{}(D2);
	\draw[->] (D2) to node [above]{}(E2);
	
	\end{tikzpicture}
	\end{equation}
	For each  injective abelian group $G$ the functor $\Ext(-;G)$ is trivial and so, there is the isomorphism:
	\begin{equation}\label{eq59}
	{\varprojlim}^{(i)} H_n (X_{\alpha};G) \simeq {\varprojlim}^{(i)} \Hom\left(H^n_N(X_{\alpha});G\right), ~~~i \ge 0.
	\end{equation}
	On the other hand, by Proposition 2 of \cite{HM}, for the direct system $ {\bar{H}^*_N({\bf X})}=\left\{ \bar{H}^n_N(X_\alpha),p_{\alpha, \alpha'}, \mathcal{A} \right\}$ we have:
	$$0 \to {\varprojlim}^1 \Hom\left(H^{n+1}_N(X_\alpha);G\right)  \to \Ext\left( \varinjlim H^n_N(X_\alpha);G\right) ~~\to ~~\varprojlim \Ext\left(H^n_N(X_\alpha);G\right)~~ \to$$
	\begin{equation}\label{eq60}
	~~ {\varprojlim}^2 \Hom\left(H^{n+1}_N(X_\alpha);G\right) ~~\to ~~0.
	\end{equation}
	By Lemma 1.3. of \cite{HM} for each  injective abelian group $G$ we obtain:
	\begin{equation}\label{eq61}
	~~ {\varprojlim}^{(i)} \Hom\left(H^{n+1}(X_\alpha);G\right)=0,~~~i \ge 1.
	\end{equation}
	Therefore, by \eqref{eq55}, \eqref{eq59} and \eqref{eq61} we obtain:
	\begin{equation}\label{eq62}
	{H}_n(X;G) \approx \varprojlim H_n(X_\alpha;G).
	\end{equation} 
	
\end{proof}

\begin{theorem} \label{thm.6}
	If ${H}_*$ is an exact  homological functor defined on the category $\mathcal{K}^2_{Top}$ of closed $P$-pairs, which isatisfies $UCF_{\mathcal{K}^2_{Pol}}$ and $CIG$ property, then it satisfies $UCF$ property for all closed $P$-pairs.
\end{theorem} 
\begin{proof} Let ${\bf p}:(X,A) \to {\bf (X,A)}=\left\{(X_\alpha, A_\beta), p_{\alpha \beta }, \alpha \in \mathcal{A}\right\} $ be a resolution. The $(X,A)$ is a closed $P$-pair and so, restriction ${\bf p}_{|A}:A \to {\bf A}=\left\{ A_\beta, p_{\alpha \alpha' |{A_\alpha}}, \alpha \in \mathcal{A}\right\} $ is a the resolution \cite{MS}. Therefore, it is sufficient to prove in the absolute case. By the  condition of the Theorem, for each resolution ${\bf p}:X \to {\bf X}=\left\{X_\alpha, p_{\alpha \alpha' }, \alpha \in \mathcal{A}\right\} $ and  an injective abelian group $G$ we have an isomorphism:
	\begin{equation}\label{eq63}
	{H}_n(X;G) \approx\varprojlim H_n(X_\alpha;G).
	\end{equation}
By the condition of the Theorem, for each $X_\alpha$  we have the exact sequence:
	\begin{equation}\label{eq64}
	0 \to \Ext\left(\bar{H}^{n+1}_N(X_{\alpha});G\right) \to H_n (X_{\alpha};G) \to \Hom\left(\bar{H}^n_N(X_{\alpha});G\right) \to 0,
	\end{equation}
	which induces the long exact sequence:
	
	\begin{equation}\label{eq65}
	\begin{tikzpicture}
	\node (A) {0};
	\node (B) [node distance=3cm, right of=A] {$\varprojlim \Ext\left(\bar{H}^{n+1}_N(X_{\alpha}),G\right)$};
	\node (C) [node distance=4cm, right of=B] {$\varprojlim H_n (X_{\alpha};G)$};
	\node (D) [node distance=4cm, right of=C] {$\varprojlim \Hom\left(\bar{H}^n_N(X_{\alpha}),G\right)$};
	\node (E) [node distance=3cm, right of=D] {};
	
	\draw[->] (A) to node [above]{}(B);
	\draw[->] (B) to node [above]{}(C);
	\draw[->] (C) to node [above]{}(D);
	\draw[->] (D) to node [above]{}(E);
	
	\node (A1) [node distance=1cm, below of=A] {};
	\node (B1) [node distance=1cm, below of=B] {${\varprojlim}^1 \Ext\left(\bar{H}^{n+1}_N(X_{\alpha}),G\right)$};
	\node (C1) [node distance=1cm, below of=C] {${\varprojlim}^1 H_n (X_{\alpha};G_0)$};
	\node (D1) [node distance=1cm, below of=D] {${\varprojlim}^1 \Hom\left(\bar{H}^n_N(X_{\alpha}),G_0\right)$};
	\node (E1) [node distance=1cm, below of=E] {};
	
	\draw[->] (A1) to node [above]{}(B1);
	\draw[->] (B1) to node [above]{}(C1);
	\draw[->] (C1) to node [above]{}(D1);
	\draw[->] (D1) to node [above]{}(E1);
	
	\node (A2) [node distance=1cm, below of=A1] {};
	\node (B2) [node distance=1cm, below of=B1] {${\varprojlim}^2 \Ext\left(\bar{H}^{n+1}_N(X_{\alpha}),G\right)$};
	\node (C2) [node distance=1cm, below of=C1] {${\varprojlim}^2 H_n (X_{\alpha};G_0)$};
	\node (D2) [node distance=1cm, below of=D1] {${\varprojlim}^2 \Hom\left(\bar{H}^n_N(X_{\alpha}),G_0\right)$};
	\node (E2) [node distance=1cm, below of=E1] {$\dots ~.$};
	
	\draw[->] (A2) to node [above]{}(B2);
	\draw[->] (B2) to node [above]{}(C2);
	\draw[->] (C2) to node [above]{}(D2);
	\draw[->] (D2) to node [above]{}(E2);
	
	\end{tikzpicture}
	\end{equation}
	Note that for each injective abelian group $G$ the functor $\Ext(-;G)$ is trivial and by \eqref{eq65} we obtain the isomorphism:
	\begin{equation}\label{eq66}
	~~ {\varprojlim}H_n(X_\alpha;G) \approx {\varprojlim} \Hom\left(\bar{H}^n_N(X_\alpha);G\right).
	\end{equation}
	If we apply the isomorphism ${\varprojlim} \Hom\left(\bar{H}^n_N(X_\alpha);G\right) \approx \Hom \left({ \varinjlim} \bar{H}^n_N(X_\alpha);G\right)$, then by \eqref{eq66} we obtain:
	\begin{equation}\label{eq67}
	~~ {\varprojlim}H_n(X_\alpha;G) \approx \Hom\left({\varinjlim} \bar{H}^n(X_\alpha);G\right)=\Hom\left( \bar{H}^n_N(X);G\right).
	\end{equation}
	Therefore, by \eqref{eq63}, if $G$ is an injective, then
	\begin{equation}\label{eq68}
	{H}_n(X;G) \approx \varprojlim \bar{H}_n(X_\alpha;G) \approx \Hom\left(  \bar{H}^n_N(X);G\right).
	\end{equation}
	Now consider any abelian group $G$ and the corresponding injective resolution:
	\begin{equation}\label{eq69}
	~~0 \to G \to G' \to G'' \to 0.
	\end{equation}
	Apply to the sequence \eqref{eq69} by the functor $\Hom\left(\bar{H}^n_N(X);-\right)$. The abelian groups $G'$ and $G''$ are injective and so we have: 	
	\begin{equation}\label{eq70}
	\begin{tikzpicture}
	\node (A) {$0$};
	\node (B) [node distance=2cm, right of=A] {$\Hom\left(\bar{H}^n_N(X);G\right)$};
	\node (C) [node distance=3cm, right of=B] {$\Hom\left(\bar{H}^n_N(X);G'\right)$};
	\node (D) [node distance=3.2cm, right of=C] {$\Hom\left(\bar{H}^n_N(X);G''\right)$};
	\node (E) [node distance=3cm, right of=D] {$\Ext\left(\bar{H}^n_N(X);G\right)$};
	\node (F) [node distance=2cm, right of=E] {$0.$};
	
	\draw[->] (A) to node [above]{}(B);
	\draw[->] (B) to node [above]{}(C);
	\draw[->] (C) to node [above]{}(D);
	\draw[->] (D) to node [above]{}(E);
	\draw[->] (E) to node [above]{}(F);	
    \end{tikzpicture}
	\end{equation}
	Therefore, for each integer $n \in N$ we have
	\begin{equation}\label{eq71}
	\Hom\left(\bar{H}^n_N(X);G\right) \simeq Ker\left(\Hom(\bar{H}^n_N(X);G'\right) \to  \Hom\left(\bar{H}^n_N(X);G'')\right),
	\end{equation}
	\begin{equation}\label{eq72}
	\Ext\left(\bar{H}^n_N(X);G\right) \simeq Coker\left(\Hom(\bar{H}^n_N(X);G'\right) \to  \Hom\left(\bar{H}^n_N(X);G''\right).
	\end{equation}
	Now apply sequence \eqref{eq69} by the homological bifunctor $\bar{H}^N_*(X;-)$, which gives the following long exact sequence:
	\begin{equation}\label{eq73}
	 \dots  \to \bar{H}^N_{n+1}(X;G') \to  \bar{H}^N_{n+1}(X;G'') \to  \bar{H}^N_{n}(X;G) \to \bar{H}^N_{n}(X;G') \to  \bar{H}^N_{n}(X;G'') \to \dots~~.
	\end{equation}
	Therefore, for each $n \in N$ we obtain the following short exact sequence:
	\begin{equation}\label{eq74}
	\begin{tikzpicture}
	\node (A) {$0$};
	\node (B) [node distance=3.5cm, right of=A] {$Coker\left(\bar{H}^N_{n+1}(X;G') \to  \bar{H}^N_{n+1}(X;G'')\right)$};
	\node (C) [node distance=4.5cm, right of=B] {$\Hom(\bar{H}^n_N(X);G)$};
	\node (D) [node distance=2cm, right of=C] {};
	\node (B1) [node distance=1cm, below of=B] {$Ker \left(\bar{H}^N_{n}(X;G') \to  \bar{H}^N_{n}(X;G'')\right)$};
	\node (C1) [node distance=3cm, right of=B1] {$0~.$};
	
	\draw[->] (A) to node [above]{}(B);
	\draw[->] (B) to node [above]{}(C);
	\draw[->] (C) to node [above]{}(D);
	\draw[->] (B1) to node [above]{}(C1);
		
	\end{tikzpicture}
	\end{equation}
	By \eqref{eq74}, \eqref{eq72} and \eqref{eq71} we obtain that for each $X \in \mathcal{K}^2_{Top}$ there is a short exact sequence:
	
	\begin{equation}\label{eq75}
	\begin{tikzpicture}
	\node (A) {$0$};
	\node (B) [node distance=2cm, right of=A] {$\Ext\left(\bar{H}^{n+1}_N(X);G\right)$};
	\node (C) [node distance=2.5cm, right of=B] {${H}_n(X;G)$};
	\node (D) [node distance=2.5cm, right of=C] {$\Hom\left(\bar{H}^n_N(X);G\right)$};
	\node (E) [node distance=2cm, right of=D] {$0$};

	\draw[->] (A) to node [above]{}(B);
	\draw[->] (B) to node [above]{}(C);
	\draw[->] (C) to node [above]{}(D);
	\draw[->] (D) to node [above]{}(E);
	\end{tikzpicture}
	\end{equation}
\end{proof}

\begin{corollary}\label{Cor.7}  The Alexander-Spanier normal homology theory $\bar{H}_*^N(-,-;G)$ defined on the category $\mathcal{K}^2_{Top}$ satisfies $CEH$, $CIG$ and $UCF$ properties.
\end{corollary}

\section{Uniqueness Theorem}

In this section we will follow the approach developed in the paper \cite{In1} to obtain the uniqueness theorem for an exact bifunctor homology theory with the $CIG$ property. We will see that the Alexander-Spanier normal homology theory $\bar{H}^N_*(-,-;G)$ is a bifunctor. On the other hand, by Theorem \ref{thm.4} and \ref{thm.5} it has $CIG$ property and so, we obtain the axiomatic characterization of $\bar{H}^N_*(-,-;G)$ on the category of pairs of general topological spaces $\mathcal{K}^2_{Top}.$

An exact homology functor $H_*(-,-;G)$ defined on the category $\mathcal{K}^2_{Top}$ is said to be a bifunctor \cite{In1}, if for each pair $(X,A) \in \mathcal{K}^2_{Top}$ and a short exact sequence
\begin{equation}\label{eq76}
0 \to G \os{\varphi}{\lra} G_1 \os{\psi}{\lra} G_2 \to 0,
\end{equation} 
there is the functorial natural long exact sequence:
\begin{equation}\label{eq77}
\dots \os{\psi _*}{\lra} {H}_{n+1}(X;G_2) \os{d_{n+1}}{\lra} {H}_n(X;G) \os{\varphi _*}{\lra} {H}_n(X;G_1) \os{\psi _*}{\lra} {H}_n(X;G_2) \os{d_n}{\lra} \dots~~.
\end{equation}
In the paper \cite{In1} Kh. Inasaridze described  an exact bifunctor homology theory using the continuity property for infinitely divisible (injective abelian)  groups on the subcategory $\mathcal{K}^2_{C}$ of compact Hausdorff pairs. In particular, it is proved that  there exists one and only one exact bifunctor homology theory on the category $\mathcal{K}^2_{C}$ of compact Hausdorff pairs with coefficients in the category of abelian groups (up to natural equivalence) which satisfies the axioms of homotopy, excision, dimension, and continuity for every infinitely divisible (injective abelian) group (see Theorem 1 in \cite{In1}).

\begin{lemma}\label{Lem.3} Each cochain complex $C^*$ and a short exact sequence of abelian groups
	\begin{equation}\label{eq78}
	0 \to G \os{\varphi}{\lra} G_1 \os{\psi}{\lra} G_2 \to 0
	\end{equation} 
induces a short exact sequence of chain complexes:
	\begin{equation}\label{eq79}
    0 \to C_*(\beta_{\#}) \os{\varphi}{\lra} C_*(\beta_{1 \#}) \os{\psi}{\lra} C_*(\beta_{2\#})\to 0.
    \end{equation}
\end{lemma}

\begin{proof} Consider an injective resolution of the exact sequence \eqref{eq78}:
	
\begin{equation}\label{eq80}
\begin{tikzpicture}

\node (A) {0};
\node (B) [node distance=1cm, right of=A] {$G$};
\node (C) [node distance=1cm, right of=B] {$G'$};
\node (D) [node distance=1cm, right of=C] {$G''$};
\node (E) [node distance=1cm, right of=D] {$0$};

\node (B0) [node distance=1cm, above of=B] {$0$};
\node (C0) [node distance=1cm, above of=C] {$0$};
\node (D0) [node distance=1cm, above of=D] {$0$};

\draw[->] (B0) to node [above]{}(B);
\draw[->] (C0) to node [above]{}(C);
\draw[->] (D0) to node [above]{}(D);

\draw[->] (A) to node [above]{}(B);
\draw[->] (B) to node [above]{$\alpha$}(C);
\draw[->] (C) to node [above]{$\beta$}(D);
\draw[->] (D) to node [above]{}(E);

\node (A1) [node distance=1cm, below of=A] {$0$};
\node (B1) [node distance=1cm, below of=B] {$G_1$};
\node (C1) [node distance=1cm, below of=C] {$G'_1$};
\node (D1) [node distance=1cm, below of=D] {$G''_1$};
\node (E1) [node distance=1cm, below of=E] {$0$};

\draw[->] (B) to node [right]{$\varphi$}(B1);
\draw[->] (C) to node [right]{$\varphi '$}(C1);
\draw[->] (D) to node [right]{$\varphi ''$}(D1);

\draw[->] (A1) to node [above]{}(B1);
\draw[->] (B1) to node [above]{$\alpha_1$}(C1);
\draw[->] (C1) to node [above]{$\beta_1$}(D1);
\draw[->] (D1) to node [above]{}(E1);

\node (A2) [node distance=1cm, below of=A1] {$0$};
\node (B2) [node distance=1cm, below of=B1] {$G_2$};
\node (C2) [node distance=1cm, below of=C1] {$G'_2$};
\node (D2) [node distance=1cm, below of=D1] {$G''_2$};
\node (E2) [node distance=1cm, below of=E1] {$0~,$};

\draw[->] (B1) to node [right]{$\psi$}(B2);
\draw[->] (C1) to node [right]{$\psi '$}(C2);
\draw[->] (D1) to node [right]{$\psi ''$ }(D2);

\draw[->] (A2) to node [above]{}(B2);
\draw[->] (B2) to node [above]{$\alpha_2$}(C2);
\draw[->] (C2) to node [above]{$\beta_2$}(D2);
\draw[->] (D2) to node [above]{}(E2);

\node (B3) [node distance=1cm, below of=B2] {$0$};
\node (C3) [node distance=1cm, below of=C2] {$0$};
\node (D3) [node distance=1cm, below of=D2] {$0$};

\draw[->] (B2) to node [right]{}(B3);
\draw[->] (C2) to node [right]{}(C3);
\draw[->] (D2) to node [right]{ }(D3);
\end{tikzpicture}
\end{equation}
which induces the mapping of exact sequences of chain complexes:

\begin{equation}\label{eq81}
\begin{tikzpicture}

\node (B)  {$\Hom\left(C^*;G'\right)$};
\node (C) [node distance=3cm, right of=B] {$\Hom\left(C^*;G''\right)$};

\node (B0) [node distance=1cm, above of=B] {$0$};
\node (C0) [node distance=1cm, above of=C] {$0$};

\draw[->] (B0) to node [above]{}(B);
\draw[->] (C0) to node [above]{}(C);

\draw[->] (B) to node [above]{$\beta_{\#}$}(C);

\node (B1) [node distance=1cm, below of=B] {$\Hom\left(C^*;G'_1\right)$};
\node (C1) [node distance=1cm, below of=C] {$\Hom\left(C^*;G''_1\right)$};

\draw[->] (B) to node [right]{$\varphi '$}(B1);
\draw[->] (C) to node [right]{$\varphi ''$}(C1);

\draw[->] (B1) to node [above]{$\beta_{1\#}$}(C1);

\node (B2) [node distance=1cm, below of=B1] {$\Hom\left(C^*;G'_2\right)$};
\node (C2) [node distance=1cm, below of=C1] {$\Hom\left(C^*;G''_2\right)$.};

\draw[->] (B1) to node [right]{$\psi '$}(B2);
\draw[->] (C1) to node [right]{$\psi '' $}(C2);

\draw[->] (B2) to node [above]{$\beta_{2\#}$}(C2);

\node (B3) [node distance=1cm, below of=B2] {$0$};
\node (C3) [node distance=1cm, below of=C2] {$0$};
\node (D3) [node distance=1cm, below of=D2] {$0$};

\draw[->] (B2) to node [right]{}(B3);
\draw[->] (C2) to node [right]{}(C3);
\draw[->] (D2) to node [right]{ }(D3);
\end{tikzpicture}
\end{equation}
Consequently, we have the exact sequence of cones of chain maps $\beta_{\#}$, $\beta_{1\#}$ and $\beta_{2\#}$ that is the sequence \eqref{eq79}.
\end{proof}

\begin{theorem}\label{thm.7}
There exists one and only one exact bifunctor homology theory (up to natural equivalence) on the category $\mathcal{K}^2_{Top},$ which satisfies the Eilenberg-Steenrod axioms on the subcategory  $\mathcal{K}^2_{Pol}$, and has $UCF_{\mathcal{K}^2_{Pol}}$ and  $CIG$ properties.
\end{theorem}

\begin{proof} Our aim is to show that any homology theory $H_*(-,-;G)$ which satisfies the conditions of the Theorem is isomorphic to the Alexander-Sapanier normal homology theory $\bar{H}^N_*(-,-;G).$

Let ${\bf \tilde{p}}:X \to {\bf \tilde{X}}=\left\{\tilde{X}_{\lambda }, p_{\tilde{\lambda} \tilde{\lambda}' }, \tilde{\lambda} \in \tilde{\Lambda} \right\} $  be a polyhedral resolution defined by normal covering of $X$ as in \cite{MS}. Note that in the paper \cite{MS} instead of ${\bf \tilde{p}}=(\tilde{p}_{\tilde{\lambda}}) : X \to {\bf \tilde{X}}$ is used ${\bf p^*}=(p^*_{\tilde{\lambda}}) : X \to {\bf X^*}$ notation. In particular, let $\Lambda$ be the set of all finite subsets $\lambda=\{\alpha_1, \alpha_2, \dots , \alpha_n\}$ consisting of different normal coverings of $X$. It is clear that $\Lambda$ is ordered by inclusion, i.e. $\lambda \le \lambda ' \iff \lambda \subset \lambda '.$ For each $\lambda \in \Lambda$ denote by $\alpha_\lambda$ the normal covering $\alpha_1 \wedge \alpha_2 \wedge \dots \wedge \alpha_n=\{ U_{\alpha_1} \bigcap U_{\alpha_2} \bigcap \dots \bigcap U_{\alpha_n} | U_{\alpha_i}\in \alpha_i, i=1,2, \dots ,n\}.$ Denote by $N_\lambda$ the nerve $N(\alpha_\lambda)$ of covering $\alpha_\lambda$.  Let $X_\lambda=|N_\lambda|$ be the geometric realization with CW-topology of $N_\lambda$. In this case, if $\lambda \le \lambda '$, then there is a uniquely defined function $p_{\lambda \lambda '}$, which maps the vertex  $U_{\alpha_1} \bigcap U_{\alpha_2} \bigcap \dots \bigcap U_{\alpha_n} \bigcap \dots \bigcap U_{\alpha_{n'}}$ of $N_{\lambda '}$ to the vertex $U_{\alpha_1} \bigcap U_{\alpha_2} \bigcap \dots \bigcap U_{\alpha_n}$ of $N_\lambda ,$ which itself induces a simplicial mapping $p_{\lambda \lambda '}:X_{\lambda '} \to X_{\lambda}.$ It is known that ${\bf X}=\left\{X_{\lambda }, p_{\lambda \lambda' }, \lambda \in \Lambda \right\}$ is an inverse system. Moreover, for each $\lambda \in \Lambda$ there is the canonical mapping $p_\lambda :X \to X_\lambda $ defined by the corresponding  partition of the unity $\Phi_\lambda=\{ \varphi_{\alpha_\lambda}\}$ of the covering $\alpha_\lambda ,$ such that
\begin{equation}\label{eq82}
p_{\lambda \lambda '}p_{ \lambda '}=p_{\lambda }, ~~~\lambda \le \lambda ' .
\end{equation}
Consequently, ${\bf p}=(p_\lambda):X \to {\bf X}$ is a mapping of the system. For each $\lambda \in \Lambda$, let $\mathcal{U}_\lambda=\left\{ U_{(\lambda, \mu)}|\mu \in \mathcal{M} \right\}$ be the system of all open neighborhoods of closure $\overline{p_\lambda\left(X\right)}$ in $X_\lambda$. Let $\tilde{\lambda}=(\lambda, \mu)$ and $\tilde{\Lambda}=\left\{ \tilde{\lambda}=(\lambda, \mu)| \lambda \in \Lambda , \mu \in \mathcal{M}\right\}.$ In this case, $\tilde{\lambda} \le^* \tilde{\lambda}'$ if  $\lambda \le \lambda'$ and $p_{\lambda \lambda '} \left(U'_{(\lambda', \mu')} \right) \subset U_{(\lambda, \mu)}.$ Let $\tilde{X}_{\tilde{\lambda}}=U_{(\lambda, \mu)}$ and $\tilde{p}_{\lambda^*}:X \to X^*_{\lambda^*}$ be the mapping $p_\lambda : X \to U_{(\lambda, \mu)} \subset X_\lambda$. Consequently, denote by $\tilde{p}_{\tilde{\lambda} \tilde{\lambda} '}:\tilde{X}_{\tilde{\lambda} '} \to X^*_{\lambda^*}$ the mapping $p_{\lambda \lambda ' |U'_{(\lambda', \mu')}}:U'_{(\lambda', \mu')} \to U_{(\lambda, \mu)}. $ It is known that  ${\bf \tilde{p}}=(\tilde{p}_{\tilde{\lambda}}) : X \to {\bf \tilde{X}}$ is a polyhedral resolution of $X$ \cite{MS}. Let ${\bf i}=\left\{1_{X_\lambda},i\right\}:{\bf X} \to {\bf \tilde{X}}$ be the mapping, where $i:\tilde{\Lambda} \to \Lambda$ is given by $i(\tilde{\lambda})=i\left(\lambda,\mu\right)=\lambda, ~~~\forall \tilde{\lambda} \in \tilde{\Lambda}$.  Consider the direct systems $\left\{\bar{C}^*_N\left(X_{\lambda }\right), p^{\#}_{\lambda \lambda' }, \lambda \in \Lambda \right\}$, $\left\{\bar{C}^*_N\left(\tilde{X}_{\tilde{\lambda} }\right), p^{\#}_{\tilde{\lambda} \tilde{\lambda}' }, \tilde{\lambda} \in \tilde{\Lambda} \right\}$ and the mapping between the corresponding limit groups
\begin{equation}\label{eq83}
{\bf i}^\#=\varinjlim {i^\#_\lambda}: \varinjlim \bar{C}^*_N\left(\tilde{X}_{\tilde{\lambda} }\right) \to \varinjlim \bar{C}^*_N\left(X_{\lambda }\right).
\end{equation}
Our aim is to show that \eqref{eq83} induces an isomorphism between the corresponding cohomology groups. Consider the following composition
\begin{equation}\label{eq84}
\varinjlim \bar{H}^*_N\left(\tilde{X}_{\tilde{\lambda} };G\right) \os{{\bf i}^*}{\lra} \varinjlim \bar{H}^*_N\left(X_{\lambda };G\right) \os{{\bf p}^*}{\lra} \bar{H}^*_N(X;G).
\end{equation}
It is clear that ${\bf \tilde{p}}={\bf i} \circ {\bf p}:X \to {\bf \tilde{X}}$ and so, the composition ${\bf \tilde{p}}^*={\bf p}^* \circ {\bf {i}}^*: \varinjlim \bar{H}^*_N\left(\tilde{X}_{\tilde{\lambda} };G\right) \to \bar{H}^*_N(X;G)$
is an isomorphism by Corollary  \ref{Cor.2'}. On the other hand, by Corollary \ref{Cor.2} and Theorem 3.2 \cite{Mor2} we have an isomorphism: 
\begin{equation}\label{eq85}
{\bf p}^*: \varinjlim \bar{H}^*_N\left(X_{\lambda };G\right) \to \bar{H}^*_N(X;G),
\end{equation}
 and therefore, ${\bf i}^*: \varinjlim \bar{H}^*_N\left(\tilde{X}_{\tilde{\lambda} };G\right) \to \varinjlim \bar{H}^*_N\left(X_{\lambda };G\right)$ must be an isomorphism as well. Therefore, by Lemma 5.5 \cite{BM3} to study the homology groups $\bar{H}^N_*(X;G)$, instead of the resolution  ${\bf \tilde{p}}=(\tilde{p}_\lambda):X \to {\bf \tilde{X}}=\left\{\tilde{X}_{\lambda }, p_{\tilde{\lambda} \tilde{\lambda}' }, \tilde{\lambda} \in \tilde{\Lambda} \right\},$ it is sufficient to consider the inverse system ${\bf X}=\left\{X_{\lambda }, p_{\lambda \lambda' }, \lambda \in \Lambda \right\}$ and the corresponding mapping  ${\bf p}=(p_\lambda):X \to {\bf X}=\left\{X_{\lambda }, p_{\lambda \lambda' }, \lambda \in \Lambda \right\}$.  

Let $\mathcal{N}(X)=\varprojlim \left\{X_{\lambda }, p_{\lambda \lambda' }, \lambda \in \Lambda \right\}$, then the mapping ${\bf p}:X \to \left\{X_{\lambda }, p_{\lambda \lambda' }, \lambda \in \Lambda \right\}$ induces a canonical map $p:X \to \mathcal{N}(X)$, i.e.,
\begin{equation}\label{eq86}
p_\lambda = p_\lambda^{\mathcal{N}(X)} \circ p, ~~~\forall ~\lambda \in\Lambda,
\end{equation}
where $ p_\lambda^{\mathcal{N}(X)}:\mathcal{N}(X) \to X_\lambda $ is a canonical projection. Consider the mapping ${\bf p}^{\mathcal{N}(X)}=\left( p_\lambda^{\mathcal{N}(X)}\right):\mathcal{N}(X) \to {\bf X}=\left\{X_\lambda, p_{\lambda \lambda'}, \lambda \in \Lambda \right\}$ and let's show that
\begin{equation}\label{eq87}
\bar{H}^*_N\left(X;G\right) \simeq  \bar{H}^*_N\left(\mathcal{N}(X);G\right).
\end{equation}
For this aim, consider the limit space $\tilde{X}= \varprojlim \left\{\tilde{X}_{\lambda }, p_{\tilde{\lambda} \tilde{\lambda}' }, \tilde{\lambda} \in \tilde{\Lambda} \right\}$ and the corresponding mapping  ${\bf \tilde{p}}^{\tilde{X}}=(\tilde{p}^{\tilde{X}}_\lambda):\tilde{X} \to {\bf \tilde{X}}=\left\{\tilde{X}_{\lambda }, p_{\tilde{\lambda} \tilde{\lambda}' }, \tilde{\lambda} \in \tilde{\Lambda} \right\}.$ Note that all terms $\tilde{X}_{\lambda }$ are polyheron and so, topologically complete. Therefore, by Theorem 6.5 and 6.16 \cite{Mar},  the mapping ${\bf \tilde{p}}^{\tilde{X}}$ is a resolution. Hence, by Corollary 3 we have the isomorphisms:
\begin{equation}\label{eq88}
\bar{H}^*_N(X;G) \simeq \varinjlim  \bar{H}^*_N (\tilde{X}_\lambda ;G),
\end{equation}
\begin{equation}\label{eq89}
\bar{H}^*_N(\tilde{X};G) \simeq \varinjlim  \bar{H}^*_N (\tilde{X}_\lambda ;G).
\end{equation}
Therefore, there is a canonical isomorphism
\begin{equation}\label{eq90}
\bar{H}^*_N(X;G) \simeq \bar{H}^*_N(\tilde{X};G).
\end{equation} 
Consider the maps $\tilde{p}:X \to \tilde{X}$, $p:X \to \mathcal{N}(X)$, $i:\mathcal{N}(X) \to \tilde{X}$ and $p^{\tilde{X}}:\tilde{X} \to \mathcal{N}(X)$ induced by  ${\bf \tilde{p}}=(\tilde{p}_\lambda):X \to {\bf \tilde{X}}=\left\{\tilde{X}_{\lambda }, p_{\tilde{\lambda} \tilde{\lambda}' }, \tilde{\lambda} \in \tilde{\Lambda} \right\},$  ${\bf p}=(p_\lambda):X \to {\bf X}=\left\{X_{\lambda }, p_{\lambda \lambda' }, \lambda \in \Lambda \right\},$ ${\bf i}=\left\{1_{X_\lambda},i\right\}:\left\{X_{\lambda }, p_{\lambda \lambda' }, \lambda \in \Lambda \right\} \to \left\{\tilde{X}_{\tilde{\lambda} }, p_{\tilde{\lambda} \tilde{\lambda}' }, \tilde{\lambda} \in \tilde{\Lambda} \right\}$ and ${\bf p}^{\tilde{X}}=(p_\lambda):\tilde{X} \to {\bf X}=\left\{X_{\lambda }, p_{\lambda \lambda' }, \lambda \in \Lambda \right\}.$ In this case $\tilde{p}=i \circ p$ and $p^{\tilde{X}} \circ i =1_{\mathcal{N}(X)}$, which induce the isomorphisms:
\begin{equation}\label{eq91}
\bar{H}^*_N(\tilde{X};G) \os{i^*}{\lra} \bar{H}^*_N(\mathcal{N}(X);G) \os{p^*}{\lra} \bar{H}^*_N(X;G),
\end{equation}
\begin{equation}\label{eq92}
\bar{H}^*_N\left(\mathcal{N}(X);G\right) \os{\left(p^{\tilde{X}} \right)^*}{\lra} \bar{H}^*_N\left(\tilde{X};G\right) \os{i^*}{\lra} \bar{H}^*_N\left(\mathcal{N}(X);G\right).
\end{equation}
By \eqref{eq90}, the composition $p^* \circ i^* : \bar{H}^*_N\left(\tilde{X};G\right) \to \bar{H}^*_N\left(X;G\right)$ is an isomorphism. On the other hand $i^* \circ \left(p^{\tilde{X}} \right)^*$ is the identity. Therefore, $i^*:\bar{H}^*_N\left(\tilde{X};G\right) \to \bar{H}^*_N\left(\mathcal{N}(X);G\right)$ is an isomorphism and by \eqref{eq90},  the isomorphism \eqref{eq87} is fulfilled. Therefore, by \eqref{eq85} and \eqref{eq87} we have:
\begin{equation}\label{eq93}
\bar{H}^*_N(\mathcal{N}(X);G) \simeq \varinjlim \bar{H}^*_N ({X}_\lambda ;G).
\end{equation}

By Theorem 6 for each topological space $X$, there is the Universal Coefficient Formula:
\begin{equation}\label{eq94}
\begin{tikzpicture}
\node (A) {$0$};
\node (B) [node distance=2.5cm, right of=A] {$\Ext \left(\bar{H}^{n+1}_N \left(\mathcal{N}\left(X\right);\mathbb{Z}\right);G\right) $};
\node (C) [node distance=3.5cm, right of=B] {${H}_n\left(\mathcal{N}\left(X\right);G\right)$};
\node (D) [node distance=3.5cm, right of=C] {$\Hom \left(\bar{H}^n_N\left(\mathcal{N}\left(X\right);\mathbb{Z} \right);G\right)$};
\node (E) [node distance=2.7cm, right of=D] {$0~.$};

\draw[->] (A) to node [above]{}(B);
\draw[->] (B) to node [above]{}(C);
\draw[->] (C) to node [above]{}(D);
\draw[->] (D) to node [above]{}(E);
\end{tikzpicture}
\end{equation}
For each $\lambda \in \Lambda$ consider the diagram induced by the natural projection $p_\lambda^{\mathcal{N}(X)} : \mathcal{N}(X) \to X_\lambda$:
\begin{equation}\label{eq95}
\begin{tikzpicture}

\node (A) {$0$};
\node (B) [node distance=2.5cm, right of=A] {$\Ext \left(\bar{H}^{n+1}_N \left(X_\lambda;\mathbb{Z}\right);G\right) $};
\node (C) [node distance=3.5cm, right of=B] {${H}_n\left(X_\lambda;G\right)$};
\node (D) [node distance=3.5cm, right of=C] {$\Hom \left(\bar{H}^n_N\left(X_\lambda;\mathbb{Z}\right);G \right)$};
\node (E) [node distance=2.7cm, right of=D] {$0~,$};

\draw[->] (A) to node [above]{}(B);
\draw[->] (B) to node [above]{}(C);
\draw[->] (C) to node [above]{}(D);
\draw[->] (D) to node [above]{}(E);

\node (A1) [node distance=1.5cm, above of=A] {0};
\node (B1) [node distance=1.5cm, above of=B] {$\Ext \left(\bar{H}^{n+1}_N \left(\mathcal{N}\left(X\right);\mathbb{Z}\right);G\right) $};
\node (C1) [node distance=1.5cm, above of=C] {${H}_n\left(\mathcal{N}\left(X\right);G\right)$};
\node (D1) [node distance=1.5cm, above of=D] {$\Hom \left(\bar{H}^n_N\left(\mathcal{N}\left(X\right);\mathbb{Z} \right);G\right)$};
\node (E1) [node distance=1.5cm, above of=E] {0};

\draw[->] (A1) to node [above]{}(B1);
\draw[->] (B1) to node [above]{}(C1);
\draw[->] (C1) to node [above]{}(D1);
\draw[->] (D1) to node [above]{}(E1);

\draw[<-] (B) to node [right]{$\Ext\left(\left(p_\lambda^{\mathcal{N}(X)}\right)^*;G\right)$}(B1);
\draw[<-] (C) to node [right]{$\left(p_\lambda^{\mathcal{N}(X)}\right)_*$}(C1);
\draw[<-] (D) to node [right]{$\Hom\left(\left(p_\lambda^{\mathcal{N}(X)}\right)^*;G\right)$}(D1);
\end{tikzpicture}
\end{equation}
which induces the following commutative diagram:
\begin{equation}\label{eq96}
\begin{tikzpicture}

\node (A) {$0$};
\node (B) [node distance=2.5cm, right of=A] {$\varprojlim \Ext \left(\bar{H}^{n+1}_N \left(X_\lambda;\mathbb{Z}\right);G\right) $};
\node (C) [node distance=3.5cm, right of=B] {$\varprojlim {H}_n\left(X_\lambda;G\right)$};
\node (D) [node distance=3.5cm, right of=C] {$\varprojlim \Hom \left(\bar{H}^n_N\left(X_\lambda;\mathbb{Z}\right);G \right)$};
\node (E) [node distance=2.5cm, right of=D] {};
\node (F) [node distance=2cm, right of=E] {${\varprojlim}^1 \Ext \left(\bar{H}^{n+1}_N \left(X_\lambda;\mathbb{Z}\right);G\right) $};
\node (G) [node distance=3cm, right of=F] {$\dots~.$};

\draw[->] (A) to node [above]{}(B);
\draw[->] (B) to node [above]{}(C);
\draw[->] (C) to node [above]{}(D);
\draw[->] (D) to node [above]{}(F);
\draw[->] (F) to node [above]{}(G);

\node (A1) [node distance=1.5cm, above of=A] {0};
\node (B1) [node distance=1.5cm, above of=B] {$\Ext \left(\bar{H}^{n+1}_N \left(\mathcal{N}\left(X\right);\mathbb{Z}\right);G\right) $};
\node (C1) [node distance=1.5cm, above of=C] {${H}_n\left(\mathcal{N}\left(X\right);G\right)$};
\node (D1) [node distance=1.5cm, above of=D] {$\Hom \left(\bar{H}^n_N\left(\mathcal{N}\left(X\right);\mathbb{Z} \right);G\right)$};
\node (E1) [node distance=1.5cm, above of=E] {0};

\draw[->] (A1) to node [above]{}(B1);
\draw[->] (B1) to node [above]{}(C1);
\draw[->] (C1) to node [above]{}(D1);
\draw[->] (D1) to node [above]{}(E1);

\draw[<-] (B) to node [right]{$\Ext\left(\left(p_\lambda^{\mathcal{N}(X)}\right)^*;G\right)$}(B1);
\draw[<-] (C) to node [right]{$\left(p_\lambda^{\mathcal{N}(X)}\right)_*$}(C1);
\draw[<-] (D) to node [right]{$\Hom\left(\left(p_\lambda^{\mathcal{N}(X)}\right)^*;G\right)$}(D1);
\end{tikzpicture}
\end{equation}
Consequently, for an injective abelian  group $G'$ we obtain the diagram:
\begin{equation}\label{eq97}
\begin{tikzpicture}

\node (C) [node distance=3.5cm, right of=B] {$\varprojlim {H}_n\left(X_\lambda;G'\right)$};
\node (D) [node distance=3.5cm, right of=C] {$\varprojlim \Hom \left(\bar{H}^n_N\left(X_\lambda;\mathbb{Z}\right);G' \right).$};

\draw[->] (C) to node [above]{}(D);

\node (C1) [node distance=1.5cm, above of=C] {${H}_n\left(\mathcal{N}\left(X\right);G'\right)$};
\node (D1) [node distance=1.5cm, above of=D] {$\Hom \left(\bar{H}^n_N\left(\mathcal{N}\left(X\right);\mathbb{Z} \right);G\right).$};

\draw[->] (C1) to node [above]{}(D1);

\draw[<-] (C) to node [right]{$\left(p_\lambda^{\mathcal{N}(X)}\right)_*$}(C1);
\draw[<-] (D) to node [right]{$\Hom\left(\left(p_\lambda^{\mathcal{N}(X)}\right)^*;G'\right)$}(D1);
\end{tikzpicture}
\end{equation}
On the other hand, $\varprojlim \Hom \left(\bar{H}^n_N\left(X_\lambda;\mathbb{Z}\right);G' \right) \simeq  \Hom \left( \varinjlim \bar{H}^n_N\left(X_\lambda;\mathbb{Z}\right);G' \right)$ and by \eqref{eq93} and \eqref{eq97} the mapping $\left(p_\lambda^{\mathcal{N}(X)}\right)_*$ is an isomorphism:
\begin{equation}\label{eq98}
\varprojlim H_n\left(X_\lambda ;G '\right) \simeq H_n(\mathcal{N}(X);G').
\end{equation}
For each injective abelian group $G'$ we have
\begin{equation}\label{eq99}
{H}_*(X;G') \simeq \varprojlim {H}_*(X_\lambda;G').
\end{equation}
Therefore, for each injective abelian  group $G',$ the map $ p:X \to \mathcal{N}(X) $ induces an isomorphism:
\begin{equation}\label{eq100}
p_*:{H}_*(X;G') \to {H}_*(\mathcal{N}(X);G'). 
\end{equation}
Consider the following diagram induced by  $ p:X \to \mathcal{N}(X) $ and an injective resolution $0 \to G \to G' \to G'' \to 0$ of a coefficient group $G$:

\begin{equation}\label{eq101}
\begin{tikzpicture}

\node (A) {$\dots$};
\node (B) [node distance=2.2cm, right of=A] {$   {H} _{n+1}(X;G') $};
\node (C) [node distance=3.2cm, right of=B] {$ {H}_{n+1}(X;G'')$};
\node (D) [node distance=2.6cm, right of=C] {$ {H}_n(X;G)$};
\node (E) [node distance=2.5cm, right of=D] {${H} _{n}(X;G') $};
\node (F) [node distance=2.8cm, right of=E] {$ {H}_ {n}(X;G'') $};
\node (H) [node distance=2.2cm, right of=F] {$\dots$~};

\draw[->] (A) to node [above]{}(B);
\draw[->] (B) to node [above]{}(C);
\draw[->] (C) to node [above]{}(D);
\draw[->] (D) to node [above]{}(E);
\draw[->] (E) to node [above]{}(F);
\draw[->] (F) to node [above]{}(H);

\node (A) {$\dots$};
\node (A1) [node distance=1.5cm, below of=A] {$\dots$};
\node (B1) [node distance=1.5cm, below of=B] {${H} _{n+1}(\mathcal{N}(X);G') $};
\node (C1) [node distance=1.5cm, below of=C] {${H} _{n+1}(\mathcal{N}(X);G'') $};
\node (D1) [node distance=1.5cm, below of=D] {${H} _{n}(\mathcal{N}(X);G) $};
\node (E1) [node distance=1.5cm, below of=E] {${H} _{n}(\mathcal{N}(X);G') $};
\node (F1) [node distance=1.5cm, below of=F] {${H} _{n}(\mathcal{N}(X);G'') $};
\node (H1) [node distance=1.5cm, below of=H] {$\dots$~.};

\draw[->] (A1) to node [above]{}(B1);
\draw[->] (B1) to node [above]{}(C1);
\draw[->] (C1) to node [above]{}(D1);
\draw[->] (D1) to node [above]{}(E1);
\draw[->] (E1) to node [above]{}(F1);
\draw[->] (F1) to node [above]{}(H1);

\draw[->] (B) to node [right]{$\simeq$}(B1);
\draw[->] (C) to node [right]{$\simeq$}(C1);
\draw[->] (D) to node [right]{$p_*$}(D1);
\draw[->] (E) to node [right]{$\simeq$}(E1);
\draw[->] (F) to node [right]{$\simeq$}(F1);

\end{tikzpicture}
\end{equation}
Therefore, by commutative diagram \eqref{eq101} the homomorphism 
\begin{equation}\label{eq102}
p_*:{H}_*(X;G) \to {H}_*(\mathcal{N}(X);G) 
\end{equation}
is an isomorphism for any coefficient abelian group $G$. Our aim is to show  that the homology group ${H}_*(\mathcal{N}(X);G)$ does not depend on the choice of a homology $H_*(-,-;G)$.  

Consider the filtration of the space $\mathcal{N}(X)$:
\begin{equation}\label{eq103}
\mathcal{N}^0(X) \subset \mathcal{N}^1(X) \subset \dots \subset \mathcal{N}^n(X) \subset \dots ~~.
\end{equation} 
Hence,
\begin{equation}\label{eq104}
\mathcal{N}^n(X)=\varprojlim \left\{X^n_{\lambda }, p_{\lambda \lambda' }, \lambda \in \Lambda \right\},
\end{equation}
where $X^n_{\lambda }$ is the $n$-dimensional skeleton of $X_{\lambda }.$ Note that, as above, we can show that there is an isomorphism as well:
\begin{equation}\label{eq105'}
	\bar{H}^*_N(\mathcal{N}^n(X);G) \simeq \varinjlim \bar{H}^*_N ({X}^n_\lambda ;G).
\end{equation}

Let's calculate some homology groups of filtrations:
\begin{enumerate}[\bf (a)]
	\item $H_p\left(\mathcal{N}^n\left(X\right);G \right)=0, ~~\forall p>n$. By the Universal Coefficient Formula we have
	\begin{equation}\label{eq105}
	\begin{tikzpicture}
	\node (A) {$0$};
	\node (B) [node distance=2.5cm, right of=A] {$\Ext \left(\bar{H}^{p+1}_N \left(\mathcal{N}^n\left(X\right);\mathbb{Z}\right);G\right) $};
	\node (C) [node distance=3.5cm, right of=B] {${H}_p\left(\mathcal{N}^n\left(X\right);G\right)$};
	\node (D) [node distance=3.5cm, right of=C] {$\Hom \left(\bar{H}^p_N\left(\mathcal{N}^n\left(X;\mathbb{Z}\right);G \right)\right)$};
	\node (E) [node distance=2.7cm, right of=D] {$0~.$};

	\draw[->] (A) to node [above]{}(B);
	\draw[->] (B) to node [above]{}(C);
	\draw[->] (C) to node [above]{}(D);
	\draw[->] (D) to node [above]{}(E);
	\end{tikzpicture}
	\end{equation}
By \eqref{eq105'}  we obtain that
	\begin{equation}\label{eq106}
\begin{tikzpicture}
\node (A) {$0$};
\node (B) [node distance=2.5cm, right of=A] {$\Ext \left(\varinjlim\bar{H}^{p+1}_N \left(X^n_\lambda ;\mathbb{Z}\right);G\right) $};
\node (C) [node distance=3.5cm, right of=B] {${H}_p\left(\mathcal{N}^n\left(X\right);G\right)$};
\node (D) [node distance=3.5cm, right of=C] {$\Hom \left(\varinjlim \bar{H}^p_N\left(X^n_\lambda;\mathbb{Z}\right);G\right)$};
\node (E) [node distance=2.7cm, right of=D] {$0~.$};

\draw[->] (A) to node [above]{}(B);
\draw[->] (B) to node [above]{}(C);
\draw[->] (C) to node [above]{}(D);
\draw[->] (D) to node [above]{}(E);
\end{tikzpicture}
\end{equation}	
On the other hand, for polyhedron the Alexander-Spanier and singular cohomology theories are isomorphic {\color{red}\cite{Sp}} and so, if $\left\{ F^n_\alpha \right\}$ is the direct system of compact subspaces of $X^n_\lambda$, then by  Theorem 3.3 \cite{BM1}, there is a short exact sequence:
\begin{equation}\label{eq107}
\begin{tikzpicture}
\node (A) {$0$};
\node (B) [node distance=2cm, right of=A] {${\varprojlim}^1 \bar{H}^{p-1}_N\left(F^n_\alpha ;\mathbb{Z}\right)$};
\node (C) [node distance=2.5cm, right of=B] {$\bar{H}^p_N\left(X^n_\lambda;\mathbb{Z}\right)$};
\node (D) [node distance=2.5cm, right of=C] {$\varprojlim\bar{H}^p_N\left(F^n_\alpha ;\mathbb{Z}\right)$};
\node (E) [node distance=2cm, right of=D] {$0~.$};

\draw[->] (A) to node [above]{}(B);
\draw[->] (B) to node [above]{}(C);
\draw[->] (C) to node [above]{}(D);
\draw[->] (D) to node [above]{}(E);
\end{tikzpicture}
\end{equation}

$F^n_\alpha$ is a compact polyhedron and so,  there is a short exact sequence:
\begin{equation}\label{eq108}
\begin{tikzpicture}
\node (A) {$0$};
\node (B) [node distance=2.3cm, right of=A] {$\Ext \left(H_{p-1}(F^n_\alpha ;\mathbb{Z}) ;G \right) $};
\node (C) [node distance=3cm, right of=B] {$\bar{H}^p_N\left(F^n_\alpha ;G\right)$};
\node (D) [node distance=3cm, right of=C] {$\Hom \left({H}_p\left(F^n_\alpha ;\mathbb{Z}\right);G\right)$};
\node (E) [node distance=2.4cm, right of=D] {$0~.$};

\draw[->] (A) to node [above]{}(B);
\draw[->] (B) to node [above]{}(C);
\draw[->] (C) to node [above]{}(D);
\draw[->] (D) to node [above]{}(E);
\end{tikzpicture}
\end{equation}
Note that $H_p(F^n_\alpha; \mathbb{Z})=0$, if $p>n$ and $H_n(F^n_\alpha; \mathbb{Z})$ is a free abelian group. Therefore, by \eqref{eq108} we have:
\begin{equation}\label{eq109}
\bar{H}^p_N\left(F^n_\alpha ;G\right) =0, ~~~\forall p \ge n+1.
\end{equation}
Therefore, by \eqref{eq107}, \eqref{eq108} and \eqref{eq109} we have
\begin{equation}\label{eq110}
H^p_N\left(X^n_\lambda ;G\right) =0, ~~~\forall p \ge n+2.
\end{equation}
If $p=n+1,$ then by \eqref{eq107} and \eqref{eq109} there is an isomorphism
\begin{equation}\label{eq111}
\bar{H}^{n+1}_N\left(X^n_\lambda ;\mathbb{Z}\right) ={\varinjlim}^1 \bar{H}^n_N\left( F^n_\alpha  ; \mathbb{Z}\right).
\end{equation}
On the other hand, by \eqref{eq108} we have the following long exact sequence:
\begin{equation}\label{eq112}
\begin{tikzpicture}
\node (A) {$0$};
\node (B) [node distance=2.5cm, right of=A] {${\varprojlim} \Ext \left(H_{n-1}\left(F^{n}_\alpha;\mathbb{Z} \right);G \right) $};
\node (C) [node distance=4cm, right of=B] {$ {\varprojlim} \bar{H}^{n}_N\left(F^{n}_\alpha ;G \right)$};
\node (D) [node distance=4cm, right of=C] { $ {\varprojlim} \Hom \left({H}_{n}\left(F^{n}_\alpha ; \mathbb{Z}\right);G\right)$};
\node (E) [node distance=2.8cm, right of=D] {};

\draw[->] (A) to node [above]{}(B);
\draw[->] (B) to node [above]{}(C);
\draw[->] (C) to node [above]{}(D);
\draw[->] (D) to node [above]{}(E);

\node (A1) [node distance=1cm, below of=A] {};
\node (B1) [node distance=1cm, below of=B] {${\varprojlim}^1 \Ext \left(H_{n-1}\left(F^{n}_\alpha ;\mathbb{Z}\right);G\right) $};
\node (C1) [node distance=1cm, below of=C] {$ {\varprojlim}^1 \bar{H}^{n}_N\left(F^{n}_\alpha ;G\right)$};
\node (D1) [node distance=1cm, below of=D] { $ {\varprojlim}^1 \Hom \left({H}_{n}\left(F^{n}_\alpha ;\mathbb{Z} \right);G\right)$};
\node (E1) [node distance=1cm, below of=E] {$\dots$~.};

\draw[->] (A1) to node [above]{}(B1);
\draw[->] (B1) to node [above]{}(C1);
\draw[->] (C1) to node [above]{}(D1);
\draw[->] (D1) to node [above]{}(E1);
\end{tikzpicture}
\end{equation}
The group $H_{n}(F^{n}_\alpha ; \mathbb{Z})$ is a finitely generated and so, by Corollary 1.5 of \cite{HM} ${\varprojlim}^1 \Hom \left({H}_{n}\left(F^{n}_\alpha ;\mathbb{Z} \right);G\right)=0$. By Proposition1.2 and  Corollary 1.5 of \cite{HM} we have
\begin{equation}\label{eq113}
{\varprojlim}^1 \Ext \left(H_{n-1}\left(F^{n}_\alpha ; \mathbb{Z} \right);G \right) \simeq {\varprojlim}^3 \Hom \left(H_{n-1}\left(F^{n}_\alpha ; \mathbb{Z}\right);G \right) =0.
\end{equation}
Therefore,  \eqref{eq113}, \eqref{eq112}, \eqref{eq111}, \eqref{eq109} and \eqref{eq106} we obtain
\begin{equation}\label{eq114}
H_{n+1}\left(\mathcal{N}^{n}(X);G\right) \simeq 0.
\end{equation}
 
\item $H_p\left(\mathcal{N}^n(X),\mathcal{N}^{n-1}(X);G\right)=0, ~~~\forall p \ne n-1,~ n.$ Consider the long exact homological sequence of the pair $\left(\mathcal{N}^n(X),\mathcal{N}^{n-1}(X) \right)$:
\begin{equation}\label{eq115}
\dots  \to {H}_{p}(\mathcal{N}^{n-1}(X);G) \to  {H}_{p}(\mathcal{N}^n(X);G) \to  {H}_{p}(\mathcal{N}^n(X),\mathcal{N}^{n-1}(X);G) \to {H}_{p-1}(\mathcal{N}^{n-1}(X);G)  \to \dots~~.
\end{equation}
Therefore, for $p >n$ by the property (a),  ${H}_{p}(\mathcal{N}^n(X);G)=0$ and so, by \eqref{eq115} we have
\begin{equation}\label{eq116}
{H}_{p}(\mathcal{N}^n(X),\mathcal{N}^{n-1}(X);G)=0, ~~~\forall p >n.
\end{equation}
For the case $p<n-1$, we use the same method as in the case (a). In particular, consider the Universal Coefficient Formula for the pair $\left(\mathcal{N}^n(X),\mathcal{N}^{n-1}(X)\right)$:

\begin{equation}\label{eq117}
\begin{tikzpicture}\small
\node (A) {$0$};
\node (B) [node distance=3.3cm, right of=A] {$\Ext \left(\bar{H}^{p+1}_N \left(\mathcal{N}^n\left(X\right),\mathcal{N}^{n-1}\left(X\right),\mathbb{Z}\right);G\right) $};
\node (C) [node distance=5cm, right of=B] {${H}_p\left(\mathcal{N}^n\left(X\right), \mathcal{N}^{n-1}\left(X\right);G\right)$};
\node (D) [node distance=5cm, right of=C] {$\Hom \left(\bar{H}^p_N\left(\mathcal{N}^n\left(X\right),\mathcal{N}^{n-1}\left(X\right);\mathbb{Z} \right);G\right)$};
\node (E) [node distance=3.3cm, right of=D] {$0~.$};

\draw[->] (A) to node [above]{}(B);
\draw[->] (B) to node [above]{}(C);
\draw[->] (C) to node [above]{}(D);
\draw[->] (D) to node [above]{}(E);
\end{tikzpicture}
\end{equation}
Using the isomorphism \eqref{eq85} the sequence \eqref{eq117} will become:
\begin{equation}\label{eq118}
\begin{tikzpicture}
\node (A) {$0$};
\node (B) [node distance=3cm, right of=A] {$\Ext \left(\varinjlim\bar{H}^{p+1}_N \left(X^n_\lambda,X^{n-1}_\lambda ;\mathbb{Z}\right);G\right) $};
\node (C) [node distance=4.5cm, right of=B] {${H}_p\left(\mathcal{N}^n\left(X\right),\mathcal{N}^{n-1}\left(X\right);G\right)$};
\node (D) [node distance=4.5cm, right of=C] {$\Hom \left(\varinjlim \bar{H}^p_N\left(X^n_\lambda , X^{n-1}_\lambda;\mathbb{Z}\right);G\right)$};
\node (E) [node distance=3cm, right of=D] {$0~.$};

\draw[->] (A) to node [above]{}(B);
\draw[->] (B) to node [above]{}(C);
\draw[->] (C) to node [above]{}(D);
\draw[->] (D) to node [above]{}(E);
\end{tikzpicture}
\end{equation}	
Consequently, if $\left\{ \left(F^n_\alpha ,E^{n-1}_\alpha \right)\right\}$ is the direct system of pairs of compact subspaces of $X^n_\lambda$ and $X^{n-1}_\lambda$, then we have:
\begin{equation}
\begin{tikzpicture}\label{eq119}
\node (A) {$0$};
\node (B) [node distance=2.5cm, right of=A] {${\varprojlim}^1 \bar{H}^{p-1}_N\left(F^n_\alpha ,E^{n-1}_\alpha ;\mathbb{Z}\right)$};
\node (C) [node distance=3.5cm, right of=B] {$\bar{H}^p_N\left(X^n_\lambda,X^{n-1}_\lambda;\mathbb{Z}\right)$};
\node (D) [node distance=3.5cm, right of=C] {$\varprojlim\bar{H}^p_N\left(F^n_\alpha ,E^{n-1}_\alpha ;\mathbb{Z}\right)$};
\node (E) [node distance=2.5cm, right of=D] {$0~,$};

\draw[->] (A) to node [above]{}(B);
\draw[->] (B) to node [above]{}(C);
\draw[->] (C) to node [above]{}(D);
\draw[->] (D) to node [above]{}(E);
\end{tikzpicture}
\end{equation}
and
\begin{equation}\label{eq120}
\begin{tikzpicture}
\node (A) {$0$};
\node (B) [node distance=2.6cm, right of=A] {$\Ext \left(H_{p-1}\left(F^n_\alpha ,E^{n-1}_\alpha  ;\mathbb{Z}\right) ;G \right) $};
\node (C) [node distance=4cm, right of=B] {$\bar{H}^p_N\left(F^n_\alpha ,E^{n-1}_\alpha  ;G\right)$};
\node (D) [node distance=4cm, right of=C] {$\Hom \left({H}_p\left(F^n_\alpha ,E^{n-1}_\alpha  ;\mathbb{Z}\right);G\right)$};
\node (E) [node distance=2.7cm, right of=D] {$0~.$};

\draw[->] (A) to node [above]{}(B);
\draw[->] (B) to node [above]{}(C);
\draw[->] (C) to node [above]{}(D);
\draw[->] (D) to node [above]{}(E);
\end{tikzpicture}
\end{equation}
By \eqref{eq120}, if $p<n-1$, then $\bar{H}^p_N\left(F^n_\alpha ,E^{n-1}_\alpha  ;G\right)=0$ and by \eqref{eq119} $\bar{H}^p_N\left(X^n_\lambda, X^{n-1}_\lambda ;G\right)=0$. Therefore, by \eqref{eq118} the given property is proven.

\item $H_{n-1}\left(\mathcal{N}^n(X),\mathcal{N}^{n-1}(X);G' \right)=0$, for an injective abelian group $G'$. By \eqref{eq117} and \eqref{eq118}, in this case we have:
\begin{equation}\label{eq121}
H_{n-1}\left(\mathcal{N}^n(X),\mathcal{N}^{n-1}(X);G'\right) \simeq  \Hom\left(\bar{H}^n_N \left(\mathcal{N}^n(X),\mathcal{N}^{n-1}(X);\mathbb{Z}\right);G' \right) \simeq \Hom\left( \varinjlim \bar{H}^{n-1}_N \left(X^n_\lambda,X^{n-1}_\lambda ;\mathbb{Z}\right);G' \right).
\end{equation}
By \eqref{eq120}  $\bar{H}^{p}_N \left(F^n_\alpha,E^{n-1}_\alpha ;\mathbb{Z}\right)=0, ~\forall p \le n-1$ and so, by \eqref{eq119} $\bar{H}^{n-1}_N \left(X^n_\lambda,X^{n-1}_\lambda ;\mathbb{Z}\right)=0.$ Therefore, by \eqref{eq121} we have 
\begin{equation}\label{eq122}
H_{n-1}\left(\mathcal{N}^n(X),\mathcal{N}^{n-1}(X);G' \right)=0.
\end{equation}

\end{enumerate}

Let $C_*\left(\mathcal{N}(X) ;G' \right)=\left\{ C_n\left(\mathcal{N}(X) ; G'\right), \partial'_n \right\}$ be the chain complex, where $C_n\left(\mathcal{N}(X) ; G'\right)=H_{n}\left(\mathcal{N}^n(X),\mathcal{N}^{n-1}(X) ;G' \right)$ for an injective abelian group $G'$, and $\partial'_n :C_n \left(\mathcal{N}(X);G'\right) \to C_{n-1}\left(\mathcal{N}(X) ; G'\right)$ be the border  mapping $\partial'_n : H_{n}\left(\mathcal{N}^n(X),\mathcal{N}^{n-1}(X) ;G' \right) \to H_{n-1}\left(\mathcal{N}^{n-1}(X),\mathcal{N}^{n-2}(X) ;G'\right)$. Note that by Universal Coefficient Formula for an injective abelian group $G'$ we have
\begin{equation}\label{eq123}
H_{n}\left(\mathcal{N}^n(X),\mathcal{N}^{n-1}(X);G'\right) \simeq \Hom\left( \varinjlim \bar{H}^{n}_N \left(X^n_\lambda,X^{n-1}_\lambda ;\mathbb{Z}\right);G' \right)
\end{equation}
and therefore, the chain complex $C_*\left(\mathcal{N}(X) ;G' \right)=\left\{ C_n\left(\mathcal{N}(X) ; G' \right), \partial_n \right\}$ is independent of choice of a homology theory. 

By the exact homological sequences for the triples $\left(\mathcal{N}^{n+1}(X),\mathcal{N}^{n}(X),\mathcal{N}^{n-1}(X)\right)$, $\left(\mathcal{N}^{n}(X),\mathcal{N}^{n-1}(X),\mathcal{N}^{n-2}(X)\right)$ and $\left(\mathcal{N}^{n+1}(X),\mathcal{N}^{n}(X),\mathcal{N}^{n-2}(X)\right)$ for an injective abelian group $G'$, we have 
 \begin{equation}\label{eq124}
	\begin{tikzpicture}\small
	\node (A) {$0$};
	\node (B) [node distance=2.5cm, right of=A]{$H_{n+1} \left(\mathcal{N}^{n+1}(X),\mathcal{N}^{n-1}(X);G'\right) $};
	\node (C) [node distance=4cm, right of=B] {$H_{n+1} \left(\mathcal{N}^{n+1}(X),\mathcal{N}^{n}(X);G'\right) $};
	\node (D) [node distance=4cm, right of=C] {$H_n \left(\mathcal{N}^{n}(X),\mathcal{N}^{n-1}(X);G'\right) $};
	\node (E) [node distance=4cm, right of=D] {$H_n \left(\mathcal{N}^{n+1}(X),\mathcal{N}^{n-1}(X);G'\right) $};
	\node (F) [node distance=2.5cm, right of=E] {$0,$};

	\draw[->] (A) to node [above]{}(B);
	\draw[->] (B) to node [above]{} (C);
	\draw[->] (C) to node [above]{$\partial'_{n+1}$} (D);
	\draw[->] (D) to node [above]{} (E);
	\draw[->] (E) to node [above]{} (F);

	\end{tikzpicture}
\end{equation}

\begin{equation}\label{eq125}
	\begin{tikzpicture}\small
	\node (A) {$0$};
	\node (B) [node distance=2.5cm, right of=A]{$H_n \left(\mathcal{N}^{n}(X),\mathcal{N}^{n-2}(X);G'\right) $};
	\node (C) [node distance=4cm, right of=B] {$H_n \left(\mathcal{N}^{n}(X),\mathcal{N}^{n-1}(X);G'\right) $};
	\node (D) [node distance=4cm, right of=C] {$H_{n-1} \left(\mathcal{N}^{n-1}(X),\mathcal{N}^{n-2}(X);G'\right) $};
	\node (E) [node distance=4cm, right of=D] {$H_{n-1} \left(\mathcal{N}^{n}(X),\mathcal{N}^{n-2}(X);G'\right) $};
	\node (F) [node distance=2.5cm, right of=E] {$0,$};

	\draw[->] (A) to node [above]{}(B);
	\draw[->] (B) to node [above]{$j'_*$} (C);
	\draw[->] (C) to node [above]{$\partial'_n$} (D);
	\draw[->] (D) to node [above]{} (E);
	\draw[->] (E) to node [above]{} (F);

	\end{tikzpicture}
\end{equation}

\begin{equation}\label{eq126}
	\begin{tikzpicture}\small
	\node (A) {$\dots$};
	\node (B) [node distance=2.5cm, right of=A]{$H_{n+1} \left(\mathcal{N}^{n+1}(X),\mathcal{N}^{n-2}(X);G'\right) $};
	\node (C) [node distance=4cm, right of=B] {$H_{n+1} \left(\mathcal{N}^{n+1}(X),\mathcal{N}^{n}(X);G'\right) $};
	\node (D) [node distance=4cm, right of=C] {$H_n \left(\mathcal{N}^{n}(X),\mathcal{N}^{n-2}(X);G'\right) $};
	\node (E) [node distance=4cm, right of=D] {$H_n \left(\mathcal{N}^{n+1}(X),\mathcal{N}^{n-2}(X);G'\right) $};
	\node (F) [node distance=2.5cm, right of=E] {$0.$};

	\draw[->] (A) to node [above]{}(B);
	\draw[->] (B) to node [above]{} (C);
	\draw[->] (C) to node [above]{$\partial '$} (D);
	\draw[->] (D) to node [above]{$q'_*$} (E);
	\draw[->] (E) to node [above]{} (F);

	\end{tikzpicture}
\end{equation}
Let consider the corresponding commutative diagram:

\begin{equation}\label{eq127}
	\begin{tikzpicture}\small
	\node (A) {$0$};
	\node (B) [node distance=2.5cm, right of=A]{$H_{n+1} \left(\mathcal{N}^{n+1}(X),\mathcal{N}^{n-1}(X);G'\right) $};
	\node (C) [node distance=4cm, right of=B] {$H_{n+1} \left(\mathcal{N}^{n+1}(X),\mathcal{N}^{n}(X);G'\right) $};
	\node (D) [node distance=4cm, right of=C] {$H_{n} \left(\mathcal{N}^{n}(X),\mathcal{N}^{n-1}(X);G'\right) $};
	\node (E) [node distance=4cm, right of=D] {$H_{n} \left(\mathcal{N}^{n+1}(X),\mathcal{N}^{n-1}(X);G'\right) $};
	\node (F) [node distance=2.5cm, right of=E] {$0$};
	\node (F1) [node distance=1.5cm, above of=F] {$0$};
	\node (D1) [node distance=1.5cm, above of=D] {$H_{n} \left(\mathcal{N}^{n}(X),\mathcal{N}^{n-2}(X);G'\right) $};
	\node (D2) [node distance=1.5cm, above of=D1] {$0$};
	\node (E1) [node distance=1.5cm, above of=E] {$H_{n} \left(\mathcal{N}^{n+1}(X),\mathcal{N}^{n-2}(X);G'\right) $};
	\node (D-1) [node distance=1.5cm, below of=D] {$H_{n-1} \left(\mathcal{N}^{n-1}(X),\mathcal{N}^{n-2}(X);G'\right) $};
	\node (D-2) [node distance=1.5cm, below of=D-1] {$H_{n-1} \left(\mathcal{N}^{n}(X),\mathcal{N}^{n-2}(X);G'\right) $};
	\node (D-3) [node distance=1.5cm, below of=D-2] {$0$};
	
	\node (C-1) [node distance=1.5cm, below of=C] {$H_{n+1} \left(\mathcal{N}^{n+1}(X),\mathcal{N}^{n-2}(X);G'\right) $};
	\node (C-2) [node distance=1.5cm, below of=C-1] {$\vdots$};
	
	\draw[->] (A) to node [above]{}(B);
	\draw[->] (B) to node [above]{} (C);
	\draw[->] (C) to node [above]{$\partial '_{n+1}$} (D);
	\draw[->] (D) to node [above]{} (E);
	\draw[->] (E) to node [above]{} (F);
	
	\draw[->] (D2) to node [right]{}(D1);
	\draw[->] (D1) to node [right]{$j'_*$}(D);
	\draw[->] (D) to node [right]{$\partial '_{n}$}(D-1);
	\draw[->] (D-1) to node [right]{}(D-2);
	\draw[->] (D-2) to node [right]{}(D-3);
	
	\draw[->] (C-2) to node [above]{}(C-1);
	\draw[->] (C-1) to node [above]{}(C);
	\draw[->] (C) to node [above]{$\partial '$}(D1);
	\draw[->] (D1) to node [above]{$q'_*$}(E1);
	\draw[->] (E1) to node [above]{}(F1);
	\end{tikzpicture}
\end{equation}
In this case we have
\begin{equation}\label{eq128}
H_{n} \left(\mathcal{N}^{n+1}(X),\mathcal{N}^{n-2}(X);G'\right) \simeq H_{n} \left(\mathcal{N}^{n}(X),\mathcal{N}^{n-2}(X);G'\right)/ Im \partial ' \simeq Im j'_* / Im(j'_* \circ \partial ') \simeq  Ker \partial '_n / Im \partial '_{n+1}.
\end{equation}
On the other, hand $Ker \partial '_n / Im \partial '_{n+1} \simeq H_n\left(C_*\left(\mathcal{N}(X);G'\right)\right)$ and so,  we have
\begin{equation}\label{eq129}
 H_n\left(\mathcal{N}^{n+1}(X), \mathcal{N}^{n-2}(X);G'\right) \simeq H_n\left(C_*\left(\mathcal{N}(X);G'\right)\right).
\end{equation}
Consider the long homological exact sequence of the triple $\left(\mathcal{N}^{n+2}(X),\mathcal{N}^{n+1}(X),\mathcal{N}^{n-2}(X)\right):$
 \begin{equation}\label{eq01}
\begin{tikzpicture}\small
\node (B) {$\dots$};
\node (C) [node distance=3cm, right of=B] {$H_{n+1} \left(\mathcal{N}^{n+2}(X),\mathcal{N}^{n+1}(X);G\right) $};
\node (D) [node distance=4.2cm, right of=C] {$H_n \left(\mathcal{N}^{n+1}(X),\mathcal{N}^{n-2}(X);G\right) $};
\node (E) [node distance=4.2cm, right of=D] {$H_n \left(\mathcal{N}^{n+2}(X),\mathcal{N}^{n-2}(X);G\right) $};
\node (F) [node distance=2.7cm, right of=E] {$0.$};

\draw[->] (B) to node [above]{} (C);
\draw[->] (C) to node [above]{$\partial$} (D);
\draw[->] (D) to node [above]{$i_*$} (E);
\draw[->] (E) to node [above]{} (F);

\end{tikzpicture}
\end{equation}
By the property (c), for an injective abelian group $G'$, we have $H_{n+1} \left(\mathcal{N}^{n+2}(X),\mathcal{N}^{n+1}(X);G'\right) =0$ and so, 
\begin{equation}\label{eq02}
H_n \left(\mathcal{N}^{n+1}(X),\mathcal{N}^{n-2}(X);G'\right) \simeq H_n \left(\mathcal{N}^{n+2}(X),\mathcal{N}^{n-2}(X);G'\right).
\end{equation}
Therefore, by \eqref{eq129} and \eqref{eq02} we have
\begin{equation}\label{eq03}
H_n \left(C_*\left(\mathcal{N}(X);G'\right)\right) \simeq H_n \left(\mathcal{N}^{n+2}(X),\mathcal{N}^{n-2}(X);G'\right).
\end{equation}

Consider the long homological exact sequence of the triple $\left(\mathcal{N}^{n+2}(X),\mathcal{N}^{n-1}(X),\mathcal{N}^{n-2}(X)\right):$
\begin{equation}\label{eq04}
\begin{tikzpicture}\small
\node (B) {$\dots$};
\node (C) [node distance=2.5cm, right of=B] {$H_{n+1} \left(\mathcal{N}^{n-1}(X),\mathcal{N}^{n-2}(X);G\right) $};
\node (D) [node distance=4cm, right of=C] {$H_{n+1} \left(\mathcal{N}^{n+2}(X),\mathcal{N}^{n-2}(X);G\right) $};
\node (E) [node distance=4cm, right of=D] {$H_{n+1} \left(\mathcal{N}^{n+2}(X),\mathcal{N}^{n-1}(X);G\right) $};
\node (F) [node distance=4cm, right of=E] {$H_{n} \left(\mathcal{N}^{n-1}(X),\mathcal{N}^{n-2}(X);G\right) $};
\node (G) [node distance=2.5cm, right of=F] {$\dots ~.$};

\draw[->] (B) to node [above]{} (C);
\draw[->] (C) to node [above]{} (D);
\draw[->] (D) to node [above]{} (E);
\draw[->] (E) to node [above]{} (F);
\draw[->] (F) to node [above]{} (G);
\end{tikzpicture}
\end{equation}
By the property (a) - $H_{p} \left(\mathcal{N}^{n}(X);G\right) =0, ~\forall p>n$ and by the long homological sequences of the pair $\left(\mathcal{N}^{n-1}(X),\mathcal{N}^{n-2}(X)\right) $ we have $H_{n+1} \left(\mathcal{N}^{n-1}(X),\mathcal{N}^{n-2}(X);G\right) =0$ and $H_{n} \left(\mathcal{N}^{n-1}(X),\mathcal{N}^{n-2}(X);G\right) =0.$  Therefore,
\begin{equation}\label{eq05}
H_{n+1} \left(\mathcal{N}^{n+2}(X),\mathcal{N}^{n-2}(X);G\right) \simeq H_{n+1} \left(\mathcal{N}^{n+2}(X),\mathcal{N}^{n-1}(X);G\right).
\end{equation} 
By \eqref{eq129} and \eqref{eq05} we have
\begin{equation}\label{eq06}
H_{n+1} \left(C_*\left(\mathcal{N}(X);G'\right)\right) \simeq H_{n+1} \left(\mathcal{N}^{n+2}(X),\mathcal{N}^{n-2}(X);G'\right).
\end{equation}
Consider the long exact homological sequence of the pair $\left(\mathcal{N}^n(X),\mathcal{N}^{n-1}(X) \right)$ for the relosution $0 \to G \os{\alpha}{\lra} G' \os{\beta}{\lra} G'' \to 0$ of an abelian group $G:$
\begin{equation}\label{eq001}
\begin{tikzpicture}\small
\node (A0) {$\dots$};
\node (A) [node distance=2.7cm, right of=A0]{$H_{n+1} \left(\mathcal{N}^{n}(X),\mathcal{N}^{n-1}(X);G\right) $};
\node (B) [node distance=4cm, right of=A]{$H_{n} \left(\mathcal{N}^{n}(X),\mathcal{N}^{n-1}(X);G\right) $};
\node (C) [node distance=4cm, right of=B] {$H_{n} \left(\mathcal{N}^{n}(X),\mathcal{N}^{n-1}(X);G'\right) $};
\node (C0) [node distance=2.5cm, right of=C] {};
\node (C1) [node distance=1.5cm, below of=A0] {};
\node (D) [node distance=4cm, right of=C1] {$H_{n} \left(\mathcal{N}^{n}(X),\mathcal{N}^{n-1}(X);G''\right) $};
\node (E) [node distance=4cm, right of=D] {$H_{n-1} \left(\mathcal{N}^{n}(X),\mathcal{N}^{n-1}(X);G\right) $};
\node (F) [node distance=4cm, right of=E] {$H_{n-1} \left(\mathcal{N}^{n}(X),\mathcal{N}^{n-1}(X);G'\right) $};
\node (G) [node distance=2.7cm, right of=F] {$\dots~~.$};

\draw[->] (A0) to node [above]{}(A);
\draw[->] (A) to node [above]{}(B);
\draw[->] (B) to node [above]{$\alpha_n$} (C);
\draw[->] (C) to node [above]{$\beta_n$} (C0);
\draw[->] (C1)+(1.5,0) to node [above]{$\beta_n$} (D);
\draw[->] (D) to node [above]{} (E);
\draw[->] (E) to node [above]{} (F);
\draw[->] (F) to node [above]{} (G);
\end{tikzpicture}
\end{equation}
By the property (a)  $H_{p} \left(\mathcal{N}^{n}(X);G\right) =0, ~\forall p>n$ and we have $H_{n+1} \left(\mathcal{N}^{n}(X),\mathcal{N}^{n-1}(X);G\right) =0.$ By the property (c), for each injective abelian group $H_{n-1} \left(\mathcal{N}^{n-1}(X),\mathcal{N}^{n-2}(X);G\right) =0.$  Therefore, by \eqref{eq001} we have the following four-term exact  sequence:
\begin{equation}\label{eq131}
\begin{tikzpicture}\small
\node (A) {$0$};
\node (B) [node distance=2.1cm, right of=A]{$H_{n} \left(\mathcal{N}^{n}(X),\mathcal{N}^{n-1}(X);G\right) $};
\node (C) [node distance=3.8cm, right of=B] {$H_{n} \left(\mathcal{N}^{n}(X),\mathcal{N}^{n-1}(X);G'\right) $};
\node (D) [node distance=3.8cm, right of=C] {$H_{n} \left(\mathcal{N}^{n}(X),\mathcal{N}^{n-1}(X);G''\right) $};
\node (E) [node distance=3.8cm, right of=D] {$H_{n-1} \left(\mathcal{N}^{n}(X),\mathcal{N}^{n-1}(X);G\right) $};
\node (F) [node distance=2.2cm, right of=E] {$0,$};

\draw[->] (A) to node [above]{}(B);
\draw[->] (B) to node [above]{$\alpha_n$} (C);
\draw[->] (C) to node [above]{$\beta_n$} (D);
\draw[->] (D) to node [above]{} (E);
\draw[->] (E) to node [above]{} (F);
\end{tikzpicture}
\end{equation}
where $\beta_n:H_{n} \left(\mathcal{N}^{n}(X),\mathcal{N}^{n-1}(X);G'\right)  \to H_{n} \left(\mathcal{N}^{n}(X),\mathcal{N}^{n-1}(X);G''\right) $ is induced by $\beta :G' \to G''$. Consequently, we obtain the chain map 
\begin{equation}\label{eq132}
\beta_{\#}=\left\{\beta_n\right\}: C_*\left(\mathcal{N}(X);G'\right) \to C_*\left(\mathcal{N}(X);G''\right).
\end{equation}
Let $C_*\left(\beta_{\#}\right)=\left\{C_n\left(\beta_{\#}\right), \tilde{\partial}_n \right\}$ be the chain cone of the chain map $\beta_{\#}$, i.e.,
\begin{equation}\label{eq133}
C_n\left(\beta_{\#}\right) \simeq C_n\left(\mathcal{N}(X);G'\right) \oplus C_{n+1}\left(\mathcal{N}(X);G''\right),
\end{equation}

\begin{equation}\label{eq134}
\tilde{\partial}_n(c'_n, c''_{n+1})=(\partial'_n(c'_n), \beta_n (c'_n) -\partial''_{n+1}(c''_{n+1})), ~~~ \forall (c'_n, c''_{n+1}) \in C_n(\beta_{\#}).
\end{equation}
Our aim is to show that
\begin{equation}\label{eq135}
H_n\left( C_*\left(\beta_\#\right)\right) \simeq H_n\left(\mathcal{N}^{n+2}(X), \mathcal{N}^{n-2}(X);G\right) .
\end{equation}
For this aim, consider the following short exact sequence:

\begin{equation}\label{eq136}
0 \to C_{*+1}\left(\mathcal{N}(X);G''\right) \os{\sigma}{\lra} C_*\left(\beta_\#\right) \os{\tau}{\lra} C_{*}\left(\mathcal{N}(X);G'\right) \to 0,
\end{equation} 
where $\sigma$ and $\tau$ are defined in the following way:
\begin{equation}\label{eq137}
\sigma \left(c''_{n+1}\right)=\left(0,c''_{n+1}\right),~~~\forall ~c''_{n+1} \in C_{n+1}\left(\mathcal{N}(X);G''\right),
\end{equation}
\begin{equation}\label{eq138}
\tau \left(c'_n,c''_{n+1}\right)=c'_n,~~~\forall ~\left(c'_n,c''_{n+1}\right) \in C_{n}\left(\beta_\#\right).
\end{equation}
Hence, the sequence \eqref{eq136} induces the long exact homological sequence:
\begin{equation}\label{eq139}
\begin{tikzpicture}\small
\node (A) {$\dots$};
\node (B) [node distance=2cm, right of=A]{$H_{n+1} \left(\mathcal{N}(X);G''\right) $};
\node (C) [node distance=3cm, right of=B] {$H_{n} \left(C_*\left(\beta_\#\right)\right) $};
\node (D) [node distance=3cm, right of=C] {$H_{n} \left(\mathcal{N}(X);G'\right) $};
\node (E) [node distance=3cm, right of=D] {$H_{n} \left(\mathcal{N}(X);G''\right) $};
\node (F) [node distance=2cm, right of=E] {$\dots~.$};

\draw[->] (A) to node [above]{}(B);
\draw[->] (B) to node [above]{$\sigma_*$} (C);
\draw[->] (C) to node [above]{$\tau_*$} (D);
\draw[->] (D) to node [above]{$E$} (E);
\draw[->] (E) to node [above]{} (F);
\end{tikzpicture}
\end{equation}
On the other hand, we have the homology exact sequence of pair $\left(\mathcal{N}^{n+2}(X),\mathcal{N}^{n-2}(X)\right)$ with respect to the short exact sequence $0 \to G \os{\alpha}{\lra} G' \os{\beta}{\lra} G'' \to 0$:
\begin{equation}\label{eq140}
\begin{tikzpicture}\small
\node (A) {$\dots$};
\node (B) [node distance=2.7cm, right of=A]{$H_{n+1} \left(\mathcal{N}^{n+2}(X),\mathcal{N}^{n-2}(X);G''\right) $};
\node (C) [node distance=4cm, right of=B] {$H_{n} \left(\mathcal{N}^{n+2}(X),\mathcal{N}^{n-2}(X);G\right) $};
\node (D) [node distance=4cm, right of=C] {$H_{n} \left(\mathcal{N}^{n+2}(X),\mathcal{N}^{n-2}(X);G'\right) $};
\node (E) [node distance=4.2cm, right of=D] {$H_{n} \left(\mathcal{N}^{n+2}(X),\mathcal{N}^{n-2}(X);G''\right) $};
\node (F) [node distance=2.7cm, right of=E] {$\dots~.$};

\draw[->] (A) to node [above]{}(B);
\draw[->] (B) to node [above]{$d_*$} (C);
\draw[->] (C) to node [above]{$\alpha_*$} (D);
\draw[->] (D) to node [above]{$\beta_*$} (E);
\draw[->] (E) to node [above]{} (F);
\end{tikzpicture}
\end{equation}
By \eqref{eq03} and \eqref{eq06}, to obtain the isomorphism \eqref{eq135} it is sufficient to define a homomorphism $\varphi : H_{n} \left(\mathcal{N}^{n+2}(X),\mathcal{N}^{n-2}(X);G\right) \to H_n\left(C_*\left(\beta_\#\right)\right)$ such that the following diagram
\begin{equation}\label{eq141}
\begin{tikzpicture}\small

\node (A) {$\dots$};
\node (B) [node distance=2.6cm, right of=A] {$H_{n+1} \left(\mathcal{N}^{n+2}(X),\mathcal{N}^{n-2}(X);G''\right) $};
\node (C) [node distance=4cm, right of=B] {$H_{n} \left(\mathcal{N}^{n+2}(X),\mathcal{N}^{n-2}(X);G\right) $};
\node (D) [node distance=4cm, right of=C] {$H_{n} \left(\mathcal{N}^{n+2}(X),\mathcal{N}^{n-2}(X);G'\right) $};
\node (E) [node distance=4cm, right of=D] {$H_{n} \left(\mathcal{N}^{n+2}(X),\mathcal{N}^{n-2}(X);G''\right) $};
\node (F) [node distance=2.5cm, right of=E] {$\dots$~};

\draw[->] (A) to node [above]{}(B);
\draw[->] (B) to node [above]{$d_*$}(C);
\draw[->] (C) to node [above]{$\alpha_*$}(D);
\draw[->] (D) to node [above]{$\beta_*$}(E);
\draw[->] (E) to node [above]{}(F);

\node (A1) [node distance=1.5cm, below of=A] {$\dots$};
\node (B1) [node distance=1.5cm, below of=B] {$H_{n+1} \left(\mathcal{N}(X);G''\right) $};
\node (C1) [node distance=1.5cm, below of=C] {$H_{n} \left(C_*\left(\beta_\#\right)\right) $};
\node (D1) [node distance=1.5cm, below of=D] {$H_{n} \left(\mathcal{N}(X);G'\right) $};
\node (E1) [node distance=1.5cm, below of=E] {$H_{n} \left(\mathcal{N}(X);G''\right) $};
\node (F1) [node distance=1.5cm, below of=F] {$\dots$~};

\draw[->] (A1) to node [above]{}(B1);
\draw[->] (B1) to node [above]{$\sigma_*$}(C1);
\draw[->] (C1) to node [above]{$\tau_*$}(D1);
\draw[->] (D1) to node [above]{$E$}(E1);
\draw[->] (E1) to node [above]{}(F1);

\draw[->] (B) to node [right]{$\simeq$}(B1);
\draw[->] (C) to node [right]{$\varphi$}(C1);
\draw[->] (D) to node [right]{$\simeq$}(D1);
\draw[->] (E) to node [right]{$\simeq$}(E1);
\end{tikzpicture}
\end{equation}
is commutative.

Consider two diagrams of type \eqref{eq127} corresponding to the coefficient groups $G'$ and $G''$ and homomorphims between them induced by $\beta:G' \to G''$. Consequently, homomorphisms and elements will be labeled with $'$ and $''$ to distinguish them. Let $\bar{h}_n \in H_{n} \left(\mathcal{N}^{n+2}(X),\mathcal{N}^{n-2}(X);G\right)$ be an element. By \eqref{eq01} there is an element $h_n \in H_{n} \left(\mathcal{N}^{n+1}(X),\mathcal{N}^{n-2}(X);G\right)$ such that $i_*(h_n)=\bar{h}_n$. Then, $\alpha_*\left(h_n\right) \in H_{n} \left(\mathcal{N}^{n+1}(X),\mathcal{N}^{n-2}(X);G'\right)$ and so by diagram \eqref{eq127} there is an element $h'_n \in H_{n} \left(\mathcal{N}^{n}(X),\mathcal{N}^{n-2}(X);G'\right)$ such that $q'_*(h'_n)=\alpha_*(h_n)$. Let $c'_n=j'_*(h'_n) \in H_{n} \left(\mathcal{N}^{n}(X),\mathcal{N}^{n-1}(X);G'\right)=C_{n} \left(\mathcal{N}(X);G'\right)$ and  $h''_n=\beta_*(h'_n) \in H_{n} \left(\mathcal{N}^{n}(X),\mathcal{N}^{n-2}(X);G''\right).$ In this case, $q''_*(h''_n)=q''_* \beta_*(h'_n)=\beta_* q'_*(h'_n)=\beta_* \alpha_*(h_n)=0.$ Therefore, $h''_n \in Ker q''_*=Im \partial ''$ and so, there exists $c''_{n+1} \in H_{n+1} \left(\mathcal{N}^{n+1}(X),\mathcal{N}^{n}(X);G''\right)=C_{n+1} \left(\mathcal{N}(X);G''\right)$ such that $\partial '' (c''_{n+1})=h''_n$. In this case,
\[\tilde{\partial} \left(c'_n,c''_{n+1}\right)=\left(\partial '_n\left(c'_n\right),\beta_*\left(c'_n\right)-\partial''_{n+1}\left(c''_{n+1}\right)\right)=\left(\partial '_n\left(j'_*\left(h'_n\right)\right),\beta_* \left(j'_*\left(h'_n\right)\right)-j''_* \left(\partial''\left(c''_{n+1}\right)\right) \right)=\]
\begin{equation}
=\left(0,j''_* \left(\beta_*\left(h'_n\right)\right)-j''_* \left(\partial''\left(c''_{n+1}\right)\right) \right)=\left(0,j''_* \left(\beta_*\left(h'_n\right)-\partial''\left(c''_{n+1}\right)\right) \right)=\left(0,j''_* \left(\beta_*\left(h'_n\right)-h''_n\right) \right)=\left(0,j''_*\left(0\right)\right)=0.
\end{equation}
Therefore
$(c'_n,c''_{n+1}) \in C_*\left(\beta_\#\right)$ is a cycle. Our aim is to show that the pair $(c'_n,c''_{n+1}) \in C_*\left(\beta_\#\right)$ is independent of chose of representatives $h_n \in H_{n} \left(\mathcal{N}^{n+1}(X),\mathcal{N}^{n-2}(X);G\right).$ Indeed, let $\tilde{h}_n \in H_{n} \left(\mathcal{N}^{n+1}(X),\mathcal{N}^{n-2}(X);G\right)$ be another representative of $\left(i_*\right)^{-1}(\bar{h}_n),$ then by  \eqref{eq01} there exists $h_{n+1} \in H_{n+1} \left(\mathcal{N}^{n+2}(X),\mathcal{N}^{n+1}(X);G\right)$ such that $\partial (h_{n+1})=h_n-\tilde{h}_n.$ Consider, the following commutative diagram:
\begin{equation}\label{eq0001}
	\begin{tikzpicture}\small
	\node (A) {$\dots$};
	\node (B) [node distance=3cm, right of=A] {$H_{n+1} \left(\mathcal{N}^{n+2}(X),\mathcal{N}^{n+1}(X);G\right) $};
	\node (C) [node distance=4.2cm, right of=B] {$H_n \left(\mathcal{N}^{n+1}(X),\mathcal{N}^{n-2}(X);G\right) $};
	\node (D) [node distance=4.2cm, right of=C] {$H_n \left(\mathcal{N}^{n+2}(X),\mathcal{N}^{n-2}(X);G\right) $};
	\node (E) [node distance=2.7cm, right of=D] {$0$};
	
	\node (A1) [node distance=1.5cm, below of=A] {$\dots$};
	\node (B1) [node distance=1.5cm, below of=B] {$H_{n+1} \left(\mathcal{N}^{n+2}(X),\mathcal{N}^{n+1}(X);G'\right) $};
	\node (C1) [node distance=1.5cm, below of=C] {$H_{n+1} \left(\mathcal{N}^{n+2}(X),\mathcal{N}^{n-1}(X);G'\right) $};
	\node (D1) [node distance=1.5cm, below of=D] {$H_{n} \left(\mathcal{N}^{n+2}(X),\mathcal{N}^{n-2}(X);G'\right) $};
	\node (E1) [node distance=1.5cm, below of=E] {$0~.$};

	\draw[->] (A) to node [above]{} (B);
	\draw[->] (B) to node [above]{$\partial$} (C);
	\draw[->] (C) to node [above]{$i_*$} (D);
	\draw[->] (D) to node [above]{} (E);
	
	\draw[->] (A1) to node [above]{} (B1);
	\draw[->] (B1) to node [above]{$\partial '$} (C1);
	\draw[->] (C1) to node [above]{$i'_*$} (D1);
	\draw[->] (D1) to node [above]{} (E1);
	
	\draw[->] (B) to node [right]{$\alpha_*$} (B1);
	\draw[->] (C) to node [right]{$\alpha_*$} (C1);
	\draw[->] (D) to node [right]{$\alpha_*$} (D1);

	\end{tikzpicture}
\end{equation}
By the property (c), we have $H_{n+1} \left(\mathcal{N}^{n+2}(X),\mathcal{N}^{n+1}(X);G'\right) =0$ and so $\alpha_* \circ \partial = \partial' \circ \alpha_*=0$. Hence,
\begin{equation}
\alpha_*(h_n)-\alpha_*(\tilde{h}_n)=\alpha_*\left(h_n-\tilde{h}_n\right)=\alpha_*\left(\partial\left(h_{n+1}\right)\right)=\partial '\left(\alpha_*\left(h_{n+1}\right)\right)=0
\end{equation}   
and so, they define the same cocyle $\left(c'_n.c''_{n+1}\right).$ Therefore, we can define $\varphi$ in the following way:
\begin{equation}\label{eq142}
\varphi(h_n)=\left(c'_n,c''_{n+1}\right)+Im \tilde{\partial}_{n+1}= \left[{\left(c'_n,c''_{n+1}\right)}\right], ~~~\forall~ h_n \in H_{n} \left(\mathcal{N}^{n+1}(X),\mathcal{N}^{n-2}(X);G'\right).
\end{equation}
Consequently, by commutativity of the diagram \eqref{eq141}, we have that for each coefficient group $G$ there is an isomorphism:
\begin{equation}\label{eq143}
H_n\left( C_*\left(\beta_\#\right)\right) \simeq H_n\left(\mathcal{N}^{n+2}(X), \mathcal{N}^{n-2}(X);G\right) .
\end{equation}

Using the long exact homological sequence of the pair $\left(\mathcal{N}^{n+2}(X), \mathcal{N}^{n-2}(X)\right)$: 
\begin{equation}\label{eq144}
\dots \to H_{n} \left(\mathcal{N}^{n-2}(X);G\right) \to H_{n} \left(\mathcal{N}^{n+2}(X);G\right) \to H_{n} \left(\mathcal{N}^{n+2}(X),\mathcal{N}^{n-2}(X);G\right) \to H_{n-1} \left(\mathcal{N}^{n-2}(X);G\right) \to \dots~~ 
\end{equation}
and  the property (a) - $H_{n}\left(\mathcal{N}^{n-2}(X);G\right) =0$, $H_{n-1}\left(\mathcal{N}^{n-2}(X);G\right)=0$ - we have an isomorphism:
\begin{equation}\label{eq145}
H_{n}\left(\mathcal{N}^{n+2}(X);G\right) \simeq  H_{n}\left(\mathcal{N}^{n+2}(X),\mathcal{N}^{n-2}(X);G\right).
\end{equation}
Note that by property (a) - $H_p\left(\mathcal{N}^{n}(X);G\right)=0, ~\forall p >n$ and  the long homological sequence of pair $\left(\mathcal{N}^{n+r+1}(X), \mathcal{N}^{n+r}(X)\right),$ $\forall r\ge 2$ we have the isomorphism
\begin{equation}\label{eq146}
H_{n}\left(\mathcal{N}^{n+r}(X);G\right) \simeq H_{n}\left(\mathcal{N}^{n+r+1}(X);G\right), \forall r \ge 2.
\end{equation} 
On the other hand, using the same method, we can prove that $H_{n}\left(\mathcal{N}(X),\mathcal{N}^{n+2}(X);G\right)=0$ and so, by the long homological sequence of pair $\left(\mathcal{N}(X), \mathcal{N}^{n+2}(X)\right)$ we have
\begin{equation}\label{eq147}
H_{n}\left(\mathcal{N}^{n+2}(X);G\right) \simeq H_{n}\left(\mathcal{N}(X);G\right).
\end{equation}  
By \eqref{eq147}, \eqref{eq146}, \eqref{eq145} and \eqref{eq143} we obtain that
\begin{equation}\label{eq148}
H_n\left(C_*\left(\beta_\#\right)\right) \simeq H_n\left(\mathcal{N}(X);G\right).
\end{equation}

\end{proof}

For a paracompact space $X$ the limit space $\mathcal{N}(X)$ is homotopy equivalent to the space $X$ (see Lemma 2.1 in \cite{San}). Therefore, the isomorphism \eqref{eq102} is fulfilled without $CIG$ property. Consequently, if we follow the proof of Theorem \ref{thm.7} we obtain the following:

\begin{corollary}[{cf. \cite[Theorem]{San}}] \label{Cor.9}
	There exists one and only one exact bifunctor homology theory (up to natural equivalence) on the category $\mathcal{K}^2_{Par}$ of closed pairs of paracompact spaces, which  satisfies the Eilenberg-Steenrod axioms on the subcategory  $\mathcal{K}^2_{Pol}$ and has $UCF_{\mathcal{K}^2_{Pol}}$  property.
\end{corollary}

Note that the given uniqueness theorem is independent of the dimension axiom. Therefore, it is true for an extraordinary homology theory as well.

\bibliographystyle{elsarticle-num}

\begin{thebibliography}{00}
	
	
	\bibitem[Ba$_1$]{Ba1} \textit{V. Baladze,} On the spectral (co)homology exact sequences of maps. Georgian Math. J. 19 (2012), no. 4, 627--638
	
	\bibitem[Ba$_2$]{Ba2} \textit{V. Baladze,} Intrinsic characterization of Alexander-Spanier cohomology groups of compactifications. Topology Appl. 156 (2009), no. 14, 2346--2356
	
	\bibitem[Bar]{Bar} \textit{V. Bartík,}	 Normal Aleksandrov-\v{C}ech cohomology, Trudy. Tbil. Math. Inst., 59 (1978), 20--49, (Russian)
	
	\bibitem[Ber-Mdz$_1$]{BM1} \textit{A. Beridze, L. Mdzinarishvili,} On the axiomatic systems of singular cohomology theory. Topology Appl. 275 (2020)
	
	\bibitem[Ber-Mdz$_2$]{BM2} \textit{A. Beridze, L. Mdzinarishvili,} On the axiomatic systems of Steenrod homology theory of compact spaces. Topology Appl. 249 (2018), 73--82
	
	\bibitem[Ber-Mdz$_3$]{BM3} \textit{A. Beridze, L. Mdzinarishvili,} On the universal coefficient formula and derivative $\varprojlim ^{(i)} $ functor, arXiv:2102.00468 
	
	\bibitem[Ber]{Ber}\textit{N. Berikashvili,} Axiomatics of the Steenrod-Sitnikov homology theory on the category of compact Hausdorff spaces.(Russian) Topology (Moscow, 1979). Trudy Mat. Inst. Steklov. 154 (1983), 24--37
	
	\bibitem[Bo-Ku]{BK}\textit{B. Botvinnik; V. Kuzminov,} A condition for the equivalence of homology theories on categories of bicompacta. (Russian) Sibirsk. Mat. Zh. 20 (1979), no. 6, 1233--1240, 1407
	
	\bibitem[Bor-Moo]{BM}\textit{Borel, A., Moore, J. C.} Homology theory for locally compact spaces. Michigan Math. J. 7 (1960), 137--159.
	
	\bibitem[Ed-Ha$_1$]{EH1} \textit{D. A. Edwards and H. M. Hastings,} \v{C}ech theory: its past, present, and future. Rocky Mountain J. Math. 10 (1980), no. 3, 429--468
	
	\bibitem[Ed-Ha$_2$]{EH2} \textit{D. A. Edwards and H. M. Hastings,} \c{Č}ech and Steenrod homotopy theories with applications to geometric topology. Lecture Notes in Mathematics, Vol. 542. Springer-Verlag, Berlin-New York, 1976
	
	
	\bibitem[Cho]{Cho}\textit{Chogoshvili, G.,} On the homology theory of topological spaces. Mitt. Georg. Abt. Akad. Wiss. USSR 1, (1940). 337--342
		
	\bibitem[Dow]{Dow}\textit{C. H. Dowker,} Homology groups of relations, Ann. of Math. (2) 56 (1952), 84--95
	
	\bibitem[H-M]{HM} {\it M. Huber; W. Meier,} Cohomology theories and infinite $CW$-complexes. Comment. Math. Helv. 53 (1978), no. 2, 239--257.
	
	\bibitem[Kol]{Kol} {\it A. N. Kolmogorff,} Les groupes de Betti des espaces localement bicompacts, C.R. Acad. Sci., Pariz 202 (1936), 1144-1147; 
	
	
	\bibitem[Kuz]{Ku}\textit{V. I. Kuzminov,} Equivalence of homology theories on categories of bicompacta. (Russian) Sibirsk. Mat. Zh. 21 (1980), no. 1, 125--129, 237.
	
	
	\bibitem[In$_1$]{In1} {\it H. N. Inasaridze,} Sakharth. SSR Mecn. Akad. Math. Inst. Šrom. 41 (1972), 128--142, 165 (1973).
	
		
	\bibitem[In$_2$]{In2} {\it Kh. Inassaridze,} On the Steenrod homology theory of compact spaces. Michigan Math. J. 38 (1991), no. 3, 323--338
	
	\bibitem[In-Mdz]{InKh} {\it Kh. Inasaridze; L. Mdzinarishvili,} On the connection between continuity and exactness in homology theory. (Russian) Soobshch. Akad. Nauk Gruzin. SSR 99 (1980), no. 2, 317--320.
	
	\bibitem[Mas]{Mas} {\it W. S. Massey,} Homology and cohomology theory. An approach based on Alexander-Spanier cochains. Monographs and Textbooks in Pure and Applied Mathematics, Vol. 46. Marcel Dekker, Inc., New York-Basel, 1978
	
	\bibitem[Mar]{Mar} {\it S. Marde$\check s$i$\acute{c}$} Strong shape and homology. Springer Monographs in Mathematics. Springer-Verlag, Berlin, 2000
	
	\bibitem[Mar-Seg]{MS} {\it S. Marde$\check s$i$\acute{c}$; J. Segal,} Shape theory. The inverse system approach. North-Holland Mathematical Library, 26. North-Holland Publishing Co., Amsterdam-New York, 1982
	
	\bibitem[Mdz$_1$]{Mdz1} {\it L. Mdzinarishvili,} On the Continuity Property of the Exact Homology Theories, Top. Proc. 56 (2020) pp. 237-247
	
	\bibitem[Mdz$_2$]{Mdz2} {\it L. Mdzinarishvili,} On homology extensions. Glas. Mat. Ser. III 21(41) (1986), no. 2, 455--482
	
	\bibitem[Mdz$_3$]{Mdz3} {\it Mdzinarishvili, L. D.} Application of the shape theory in the characterization of exact homology theories and the strong shape homotopic theory. Shape theory and geometric topology (Dubrovnik, 1981), pp. 253--262, Lecture Notes in Math., 870, Springer, Berlin-New York, 1981
	
	\bibitem[Mil]{Mil} {\it J. Milnor,} On the Steenrod homology theory, Mimeographed Note, Princeton, 1960, in: Novikov Conjectures, Index 	Theorems and Rigidity, vol. 1, in: Lond. Math. Soc. Lect. Note Ser., vol. 226, Oberwolfach, 1993, pp. 79–96.
	
	
	\bibitem[Mor$_1$]{Mor} {\it K. Morita,}
	\v{C}ech cohomology and covering dimension for topological spaces. Fund. Math. 87 (1975), 31--52.
	
	\bibitem[Mor$_2$]{Mor2} {\it K. Morita,}
	On shapes of topological spaces. Fund. Math. 86 (1975), no. 3, 251–259.

     \bibitem[Mrz$_1$]{Mrz1} {\it P. Mrozik.} Generalized Steenrod homology theories are identical with partially continuous homology theories. Homology Homotopy Appl. 20 (2018), no. 2, 61--77
     		
	\bibitem[Mrz$_2$]{Mrz2} {\it P. Mrozik.} Extending functors from the category of strict morphisms of inverse systems to the associated pro-category with applications to the first derived limit. Tsukuba J. Math. 41 (2017), no. 2, 265--296
	

	
	\bibitem[San]{San} {\it S. Saneblidze,} A uniqueness theorem for homology theory of paracompact spaces, Trudy.Tbil. Math.Inst., 97 (1992), 53--64, (Russian)
	
	\bibitem[Sit]{Sit} {\it K. A. Sitnikov,} Combinatorial topology of nonclosed sets. II. Dimension. (Russian) Mat. Sb. N.S. 37(79) (1955), 385--434.
	
	\bibitem[Skl]{Skl} {\it E.G. Sklyarenko,} The homology and cohomology of general spaces. (Russian) Current problems in mathematics. Fundamental directions, Vol. 50 (Russian), 129--266, Itogi Nauki i Tekhniki, Akad. Nauk SSSR, Vsesoyuz. Inst. Nauchn. i Tekhn. Inform., Moscow, 1989.
	
	\bibitem[Sp]{Sp} {\it E. H. Spanier,} Algebraic Topology. Corrected reprint of the 1966 original. New York: Springer-Verlag, 1966.
	
	\bibitem[St]{St} {\it N. E. Steenrod,} Regular cycles of compact metric spaces, Ann. of Math. (2). 1940. V. 41. P. 833--851.
	
	\bibitem[Wat]{Wat} {\it T. Watanabe,} The continuity axiom and the \v{C}ech homology, Geometric topology and shape theory (Dubrovnik, 1986), 221--239, Lecture Notes in Math., 1283, Springer, Berlin, 1987.  
	 
	\end{thebibliography}

\end{document}